\documentclass[3p, times]{elsarticle}
\makeatletter
\def\ps@pprintTitle{%
  \let\@oddhead\@empty
  \let\@evenhead\@empty
  \def\@oddfoot{}
  \let\@evenfoot\@oddfoot
}
\makeatother
\usepackage{amsmath}    
\usepackage{amsfonts}
\usepackage{mathtools}
\usepackage{srcltx}
\usepackage{eurosym}
\usepackage{amssymb}
\usepackage{amsthm}
\usepackage{graphicx}
\usepackage{mathrsfs}
\usepackage{xcolor}
\usepackage{exscale}
\usepackage{latexsym}
\usepackage{enumerate}
\usepackage[shortlabels]{enumitem}
\usepackage{bookmark}
\usepackage{wasysym}
\usepackage{geometry}
\usepackage{esint}
\usepackage{float}
\usepackage{comment}
\usepackage{multirow}
\usepackage{framed}
\usepackage{tikz}
\usetikzlibrary{angles, quotes}
\usepackage{adjustbox}
\usepackage{seqsplit}
\usepackage{hyperref}
\allowdisplaybreaks

\usepackage{enumitem}
\setlist[enumerate]{leftmargin=*}

\usepackage[toc,page]{appendix}
\usepackage{longtable}
\newcommand\numberthis{\addtocounter{equation}{1}\tag{\theequation}}

\theoremstyle{plain}
\newtheorem{thm}{Theorem}[section]

\newtheorem{lemma}[thm]{Lemma}

\newtheorem{proposition}[thm]{Proposition}
\newtheorem{claim}{Claim}

\theoremstyle{remark}
\newtheorem{note}{Remark}

\begin{document}   
\begin{frontmatter}

\title{\texorpdfstring{Irreducibility of Generalised Laguerre Polynomials for Large Integer Parameters}{Irreducibility of Generalised Laguerre Polynomials for Large Integer Parameters}\tnoteref{titlefootnote}}

\author[bits]{Drishya T. Das}
\ead{p20230030@goa.bits-pilani.ac.in}

\author[isi]{Saheli Dutta}
\ead{saheli21r@isid.ac.in}

\author[isi]{Shanta Laishram}
\ead{shanta@isid.ac.in}

\author[bits]{Saranya G. Nair\corref{cor1}}
\ead{saranyan@goa.bits-pilani.ac.in}

\cortext[cor1]{Corresponding author}

\affiliation[bits]{organization={BITS Pilani K K Birla Goa Campus},
            addressline={NH 17B, Zuarinagar}, 
            city={Goa},
            postcode={403726}, 
            state={Goa},
            country={India}}

\affiliation[isi]{organization={Indian Statistical Institute, Delhi Centre},
            addressline={7, S. J. S. Sansanwal Marg}, 
            city={New Delhi},
            postcode={110016}, 
            state={Delhi},
            country={India}}
\begin{abstract}
    For a positive integer $n$, we study the irreducibility of the generalised
    Laguerre polynomials
    \[
    L_n^{(a)}(x)
    = \frac{1}{n!}\sum_{j=0}^{n} \binom{n}{j}
    \frac{(n+a)!}{(j+a)!}(-1)^j x^j,
    \]
    where $a \in \mathbb{Z}_{\ge 0}$. Earlier works established that these polynomials
    are irreducible over $\mathbb{Q}$ for $0 \le a \le 50$. In this paper, we substantially strengthen these results by establishing that  $L_n^{(a)}(x)$ are
    irreducible for all $a \le 1000$. We derive sharper constraints on potential integer
    roots and make systematic use of the geometry of the associated
    Newton polygons. 
    This approach allows us to extend the irreducibility range significantly while keeping explicit computations to a minimum.
    \end{abstract}
    \begin{keyword}
    Irreducibility\sep generalised Laguerre polynomials\sep Newton polygons\sep Primes

    \MSC[2020] 33Cxx\sep 33C45\sep 11R09\sep 11A41\sep 11N05
    \end{keyword}
    \end{frontmatter}
    
\section{Introduction}
    For $a \in \mathbb R$ and $n \in \mathbb Z^+$, the generalised Laguerre polynomials of degree $n$ is defined as  \begin{align*}
        L_n^{(a)}(x)=\frac{1}{n!}\sum\limits_{j=0}^{n}\binom nj(n+a)(n+a-1)\cdots(j+a+1)(-1)^jx^j.
    \end{align*}
    These polynomials were studied in detail by P\`{o}lya and Szeg\`{o} in their classic and celebrated work \cite{polya_szego}. Although they are primarily known for their analytic and orthogonality properties, they also exhibit interesting algebraic and combinatorial behaviour. The algebraic aspect was first investigated by I. Schur \cite{schur_1929_1}, who established irreducibility results for certain Laguerre polynomials. Schur's work initiated a rich line of investigation into the factorisation and irreducibility of generalised Laguerre polynomials and related families. Throughout this paper, irreducibility means irreducibility over \(\mathbb{Q}\).

    I. Schur proved that $L_n^{(1)}(x)$ is irreducible for every positive integer $n$. Moreover, he showed that its associated Galois group is the alternating group $A_n$ when $n>1$ is odd, and the symmetric group $S_n$ otherwise. He also proved that the Galois group associated with $L_n^{(-n-1)}(x)$ is $A_n$ if $n\equiv 0 \pmod{4}$, and $S_n$ otherwise. Later, Gow \cite{gow} showed that if $n>2$ is even, then the Galois group associated with $L_n^{(n)}(x)$ is $A_n$, provided that $L_n^{(n)}(x)$ is irreducible. Filaseta, Kidd and Trifonov \cite{filaseta_kidd_trifonov} proved that for every integer $n>2$ with $n\equiv 2 \pmod{4}$, the polynomial $L_n^{(n)}(x)$ is irreducible; whereas for $n\equiv 0 \pmod{4}$, either $L_n^{(n)}(x)$ is irreducible or it factors as a linear polynomial times an irreducible polynomial of degree $n-1$. For $n=2$, the polynomial $L_2^{(2)}(x)$ is reducible, but its Galois group is $A_2$, the trivial group. These results explicitly settle the inverse Galois problem for alternating groups: for every positive integer $n>1$, there exists an explicit Laguerre polynomial of degree $n$ whose Galois group is $A_n.$

    In another direction, Filaseta and Lam \cite{filaseta_lam} proved in 2002 that, for a fixed rational number $a$ which is not a negative integer, $L_n^{(a)}(x)$ is irreducible for all but finitely many $n$. More recently, Banerjee \cite{banerjee} showed that for every even integer $n\equiv 0 \pmod{4}$ with $n\geq n_0$, the polynomial $L_n^{(n)}(x)$ has no linear factor, where $n_0$ is an effectively computable constant. Consequently, in view of the result of Filaseta, Kidd and Trifonov \cite{filaseta_kidd_trifonov}, it follows that $L_n^{(n)}(x)$ is irreducible for every \(n\equiv 0 \pmod{4}\) with \(n\geq n_0\). Explicit irreducibility results have also been obtained for $L_n^{(a)}(x)$ when $a=u+\frac{\alpha}{d},$ where \(u,\alpha\in\mathbb{Z}\), with \(\alpha=0\) if \(d=1\), and \(1\leq \alpha<d\), \(\gcd(\alpha,d)=1\), if \(d>1\). Such results are known for \(d\leq 4\) and for various ranges of the integer parameter \(u\). For a survey of these results, please see \cite{nair_shorey_survey}.
        
   In this article, we restrict our attention to the integral-parameter case, namely \(d=1\), so that \(a\in\mathbb{Z}.\) Throughout the paper, we assume that \(n\in\mathbb{Z}^{+}\), \(1\leq k\leq n/2\) and $a\in\mathbb Z$.
 The irreducibility of $L_n^{(a)}(x)$ for $a=0,1$ was proved by Schur in a series of papers \cite{schur_1931,schur_1973_1}. Subsequently, Filaseta, Finch and Leidy in \cite{filaseta_finch_leidy} proved the irreducibility of $L_n^{(a)}(x)$ for $0\leq a\leq 10$ except when $(n,a) \in \{(2,2),(4,5),(2,7)\}$. Later, Laishram and Shorey in \cite{laishram_shorey_integer} extended this result for $0\leq a\leq 50$ with only finitely many explicitly determined exceptional pairs $(n,a)$. In this article, we extend the irreducibility result of $L_n^{(a)}(x)$ to $0\leq a\leq 1000$. We now state our main results.

    \begin{thm}\label{main}
        Let $0\leq a\leq 1000$. Then, $L_n^{(a)}(x)$ is irreducible over $\mathbb Q$ except when $n = 2$, $a\in\{m^2-2:m\in\mathbb Z^+, \ 2\leq m\leq 31\}$ and $n = 4, a\in \{5,23\}$. 
        In these exceptional cases, $L_n^{(a)}(x)$ has a linear factor.
    \end{thm}

    \begin{thm}\label{n=2,3}
        Let $a\in\mathbb Z^+$. Then the following holds: 
        \begin{enumerate}[$(i)$]
        \item $L_2^{(a)}(x)$ has a linear factor if and only if $a+2$ is a perfect square.
        \item  If $L_3^{(a)}(x)$ has a linear factor, then $a+3=m^3$ if $a$ is even and $a+3=2\cdot m^3$ or $a+3=2^2\cdot m^3$ if $a$ is odd, for some $m\in\mathbb N$. 
        \end{enumerate}
    \end{thm}  

    We denote
    $$\mathcal{L}_n^{(a)}(x)=n!L_n^{(a)}(x).$$
    Note that ${L}_n^{(a)}(x)$ has a factor of degree $k$ if and only if $\mathcal{L}_n^{(a)}(x)$ has a factor of degree $k$. 
    In earlier works, several number-theoretic methods, beginning with the ideas of Schur and later developed through the use of $p$-adic valuations, Newton polygons, and results on the distribution of prime divisors, were used to establish the irreducibility of generalised Laguerre polynomials for various ranges of the parameter $a$. One of the main difficulties in extending these results to larger values of $a$ is the possible occurrence of the linear factors. In this paper, we build on this framework and refine the underlying number-theoretic analysis. Our first key observation is that if $\mathcal{L}_n^{(a)}(x)$ has a linear factor, then its integral roots lie in a much smaller explicitly describable set depending on $n$ and $a$. 
   \begin{proposition}\label{root}
    Let $n,a\in\mathbb Z^+$, $n\geq 2$, $a\geq 4$ and 
    let $\xi\in\mathbb Z$ be an integer root of $\mathcal L_n^{(a)}(x)$. Then $n\mid \xi^n$, $(n+a)\mid \xi^n$, and there exist integers $d_1, d_2$ such that $
    d_1\mid (n-1), d_2\mid (n+a-1)$ and
    \begin{align}\label{equation for root}
    \xi
    =
    2n+a-1
    +
    \frac{(n-1)(n+a-1)}{d_1 d_2}\,\gamma,
    \quad
    \gamma\in\mathbb Q.
    \end{align}
    Further, if $\gcd(\frac{n-1}{d_1},\frac{n+a-1}{d_2})=1$, then $\gamma\in \mathbb Z$ and if  $\gamma\neq 0$, then $
    d_1 d_2 > \frac{n-1}{6}.$
    \end{proposition}
    
  In addition, we extend a well-known irreducibility criterion of Filaseta \cite{filaseta_all_but_finite_bessel}, widely used in this context, by incorporating information from multiple edges of the associated Newton polygon, rather than relying only on the right-most edge, as given in Section $4$. Together, these refinements allow us to extend the known irreducibility range to all integers with $ 0 \leq a \le 1000$,  without extensive computation, relying instead on sharper bounds coming from $p$-adic valuations and the geometry of Newton polygons.

    We introduce the preliminaries in Section 2 and prove Theorem \ref{n=2,3} in Section 3.  The main irreducibility criterion required for the proof of Theorem \ref{main} is given in Section 4. In Section 5, we prove Theorem \ref{main} for $k\geq 2$, and in Section 6, we prove Proposition \ref{root} and complete the proof of Theorem \ref{main}. All the calculations have been carried out using SageMath. 
 
\section{Preliminary Lemmas} 
    Let $m$ be a positive integer. We define $\lceil m\rceil$ to be the smallest integer greater than or equal to $m$. Let $p$ be a prime. We define the $p$-adic valuation $\nu_p(m)$ to be the unique integer $e \ge 0$ such that $p^{e}\mid m$ but $p^{e+1}\nmid m$; we write $\nu_p(m)=e$ and $p^e\parallel m$. We also set $\nu_p(0)=\infty$.

    Let $\omega(m)$ denote the number of distinct prime divisors of $m$. Given a prime $p$, we define
    \begin{align}\label{omega}
    \omega_p(m)=\#\{\,q : q \mid m, q\text{ prime } \text{ and } q \ge p\,\}.
    \end{align}
    Let $P(m)$ denote the largest prime factor of $m$, with the convention that $P(1)=1$. Further, for $l\in \mathbb Z^+$, $$\Delta(m,l):=m(m+1)\cdots(m+l-1)\quad \text{and}\quad \operatorname{Rad} (m):=\prod\limits_{p\mid m,\ p\text{ prime}}p.$$
    
    The results below give useful lower bounds for \(P(\Delta(m,k))\) under suitable hypotheses. The following result is due to Laishram and Shorey \cite[Theorem 1(a)]{laishram_shorey_GPF}. 

    \begin{lemma}\label{2k+13}
    If $k \ge 2$ and $m \ge \max(2k+13,\, \tfrac{541}{262}\,k)$, then $P(\Delta(m-k+1,k)) > 2k.$
    \end{lemma}
   The following result is due to Nair and Shorey \cite[Theorem 2]{nair_shorey_GPF}.

    \begin{lemma}\label{4.42k}
     Let $m >4k$, $k\geq 2$ be such that $(m,k) \notin T$, where 
     \begin{align*}
        T:=&\{ (m,2):m=9,14,15,20,24,27,35,48,63,80,125,224,2400,4374 \} \cup \\
        &\{(m,3):m=13,14,20,24,25, 26, 48, 54, 63, 64, 98, 350\} \cup \\
        &\{(m,4): m=24,25,32,33,48,49,63\} \cup \{(m,5):m=24,32, 48\} \cup \{(m,7):m=29,30\}.
        \end{align*}
     
     Then $P(\Delta(m,k)) >4.42k,\ \text{and moreover}\  P(\Delta(m,k)) >4.5k$ unless $k \in \{82,83\}$.
    \end{lemma}

   C. St{\o}rmer \cite{stormer} proved that, for every fixed \(n\in\mathbb{Z}^{+}\), the set of integers \(m\) satisfying \(P(m(m-1))\leq n\) is finite. Later, Lehmer \cite[Table IA]{lehmer} listed all such integers \(m\) for each prime bound \(p\leq 41\). He further provided, in Table IB, the factorisation of \(m(m-1)\) for all integers \(m\geq 10^{5}\) with \(P(m(m-1))\leq 41\). The following lemma follows from a table given by  Najman in \cite{najman}.

    \begin{lemma}\label{table-najman}
    For any integers $m>31$, $m'>37$, $P(\Delta(m,6))>31$ and $P(\Delta(m',6))>37$.
    \end{lemma}
    We shall use the following theorem of Lucas'\cite{lucas}, commonly known as Lucas' theorem together with an immediate consequence.
    \begin{lemma}\label{lucas}
        Let $m$ and $m'$ be non-negative integers with $m\geq m'$ and let $p$ be a prime. Write
        $m=\sum\limits_{i=0}^{r}m_ip^i$ and $m'=\sum\limits_{i=0}^{r}m'_ip^i$ for their base-$p$ expansions. Then
        $ \binom{m}{m'} \equiv \prod\limits_{i=0}^{r} \binom{m_i}{m'_i} \pmod{p},$
        where we use the convention that \(\binom{u}{v}=0\) if \(u<v\). Consequently, if \(p\mid m\) and \(p\nmid m'\), then $\nu_p\left(\binom{m}{m'}\right)\geq 1.$
    \end{lemma}
    We shall use the following classical result due to Legendre \cite[Ch. 17, p. 263]{Hasse}. 
\begin{lemma}\label{sigma}
    Let $m\in\mathbb Z^+$ and $p$ be a prime. Then $\nu_p(m!)=\frac{m-\sigma_p(m)}{p-1},$
    where $m=\sum_{i=0}^{r}m_ip^i$ is the base-\(p\) expansion of \(m\), with \(0\leq m_i\leq p-1\) and $\sigma_p(m)=\sum_{i=0}^{r}m_i.$
\end{lemma}
 \begin{lemma}\label{equationmm'}
    Let $N\leq 1005$ be a positive integer. Then the following assertions hold:
    \begin{enumerate}[$(i)$]
        \item There do not exist primes $p$ and $q,$ with $7 \leq p, q \leq 31,$ such that $p^2\mid N \ \text{and} \ q^2\mid(N-1).$
        \item The only triples \((N,p,q)\), where $N\leq 1005$ and $p,q\geq 11$ are primes, satisfying $p^2\mid N \ \text{and} \ q^2\mid(N-2)$ are $(N, p, q)\in\{(847, 11, 13), (363, 11, 19)\}.$
        \item There do not exist a prime $p\geq 7$ such that $p^2\mid(N-2)$ and $5^3\mid(N(N-1))$.  
    \end{enumerate}
       
    \end{lemma}
    \begin{proof}
        For $(i),$ suppose that $N=mp^2$ and $N-1=m'q^2$, where $p$ and $q$ are primes satisfying \(7\leq p,q\leq 31\). Since \(N\leq 1005\), we have $1\leq m\leq \left\lfloor\frac{1005}{p^2}\right\rfloor$ and $1\leq m'\leq \left\lfloor\frac{1004}{q^2}\right\rfloor$. Then $mp^2-m'q^2=1$. By computation, we verify that there are no $m, m', p, q$ in the given range that satisfy this equation. The proof of $(ii)$ and $(iii)$ follows similarly by computation.
    \end{proof}
   \subsection{Newton polygon}
    Let $f(x) = \sum\limits_{j=0}^{n} a_j x^j \in \mathbb{Z}[x]$ with \( a_0 a_n \ne 0 \). Let $p$ be a prime and 
    $$S := \{ (0, \nu_p(a_n)), (1, \nu_p(a_{n-1})), \ldots, (n-1, \nu_p(a_1)), (n, \nu_p(a_0)) \},$$ a set of points in the extended plane. Then the Newton polygon of $f(x)$ with respect to $p$ is the lower edges of the convex hull formed by the points in $S$ starting from $(0,\nu_p(a_n))$ and ending at $(n,\nu_p(a_0))$ and is denoted by $NP_p(f(x))$. We first state an immediate consequence of Dumas' theorem on the Newton polygon of a product of polynomials\cite{dumas}.

    \begin{lemma}\label{cor-dumas}
    Let $f(x)\in \mathbb{Z}[x]$ be a polynomial of degree $n$ with $f(0)\neq 0$. Let $p$ be a prime. Suppose $f(x)=g_1(x)g_2(x)\cdots g_r(x)$ for $g_i(x)\in \mathbb{Z}[x]$ (are not necessarily irreducible) for $1\le i \le r$. Let $S_1= \{ (\Tilde{x_0},y_0), (\Tilde{x_1},y_1),(\Tilde{x_2},y_2),\ldots,(\Tilde{x_t},y_t) \}$ be the complete set of lattice points with $0=\Tilde{x_0}< \Tilde{x_1} < \Tilde{x_2} < \cdots < \Tilde{x_t}$. Let $b_j=\Tilde{x_j}-\Tilde{x}_{j-1}$ for $1 \le j \le t$. For each $j$ there exists exactly one $i$ such that $\epsilon_{i,j}\in \{0,1\}$ and $\deg(g_i)=\sum_{j=1}^{t}\epsilon_{i,j}b_j$.
    \end{lemma}
    Using the same theorem of Dumas, Filaseta \cite[Lemma 2]{filaseta_all_but_finite_bessel} established the following useful irreducibility criterion.

    \begin{lemma}\label{fil1}
        Let $k$ be a positive integer such that $k\leq \frac{n}{2}$, $g(x)=\sum\limits_{j=0}^n a_jx^j\in\mathbb Z[x]$ and $p$ be a prime such that $p\nmid a_n$, $p\mid a_j$ for all $j\in\{0,1,\cdots,n-k\}$, and the right-most edge of the Newton polygon for $g(x)$ with respect to $p$ has slope $<\frac1k$. Then for any integers $c_0,c_1,\cdots,c_n$ with $|c_0| = |c_n| =1$, the polynomial $f(x)=\sum\limits_{j=0}^n a_jc_jx^j$ cannot have a factor in $\mathbb Z[x]$ of degree $k$.
    \end{lemma}
    We consider
    $\mathcal{L}_n^{(a)}(x)=n!L_n^{(a)}(x)=\sum\limits_{j=0}^{n}(-1)^j a_jx^j,$ where $a_j=\binom{n}{j}\frac{\Delta(a+1,n)}{\Delta(a+1,j)}=\binom{n}{j}(n+a)(n+a-1)\cdots(j+a+1)=n(n-1)\cdots(j+1)\binom{n+a}{j+a}.$
    Let $n>n-r_1>n-r_2>\cdots>0$ be the $x$-coordinates of the vertices of $NP_p\left(\mathcal{L}_n^{(a)}(x)\right)$, arranged from right to left. Thus $r_1$ is the horizontal length of the right-most edge. For $j>j'$, we denote by $\phi_{(j,j')}(p)$ the slope of the edge joining the points $(n-j, \nu_p(a_j))$ and $(n-j',\nu_p(a_{j'}))$. Then 
        \begin{align}\label{phijj'}
        \phi_{(j, j')}(p)&= \frac{1}{(j-j')}\nu_p\left(\frac{\binom{n}{j'}(a+j)!}{\binom{n}{j}(a+j')!}\right).
        \end{align}
    \noindent If $j'=0$, we simply denote $\phi_{(j,0)}(p)$ by $\phi_j(p)$.  
        We denote by $\phi^{(i)}$ the slope of the $i^{th}$ edge from the right. Then the slope of the edge joining $(n-j, \nu_p(a_{j}))$ and $(n, \nu_p(a_0))$ is \begin{align}\label{phij}
         \phi_j(p)=\phi_{(j,0)}(p)=\frac{\nu_p(\Delta(a+1,j))-\nu_p\left(\binom nj\right)}{j},\end{align} and the slope of the right-most edge of $NP_p\left(\mathcal{L}_n^{(a)}(x)\right)$ is \begin{align}\label{phix1}
        \phi^{(1)}=\phi_{r_1}(p)=&\max\limits_{1\leq j\leq n}\phi_j(p).
        \end{align}
    \begin{lemma}\label{ndivaj}
        Let $n,a\in \mathbb Z^+$, $1\leq k\leq\frac{n}{2}$ and $\mathcal{L}_n^{(a)}(x)=\sum\limits_{i=0}^n (-1)^ja_jx^j$. Then the following hold:
        \begin{itemize}
            \item [(i)] For $0\leq j\leq n-1$, $n\mid a_j$ and $(n+a)\mid a_j,$ while for $0\leq j\leq n-2$, $(n-1)\mid a_j$ and $(n+a-1)\mid a_j$.
            \item[(ii)] If $p$ is a prime dividing $\Delta(n-k+1, k)\cdot\Delta(n+a-k+1,k)$, then $p\nmid a_n$ and $p\mid a_j$ for all $0\leq j\leq n-k$. 
              \end{itemize}
    \end{lemma} 
    \begin{proof}
    The proof follows directly from the definition of $a_j$.
    \end{proof}
   Laishram and Shorey obtained the following factor-exclusion results in \cite{laishram_shorey_integer}.

 \begin{lemma}\label{previous}
    Let  $1 \leq k \leq \dfrac{n}{2}$. Then the following hold: 
    \begin{itemize}
        \item[(i)] Let $a > 0$, and $u_0 = \dfrac{a}{k}$. If there is a prime $p \geq k+2$ with
    $p \mid \Delta(n-k+1,k)\cdot\Delta(n+a-k+1,k),$ and $p \nmid \Delta(a+1, k)$ satisfying $p \geq \min(2u_0, k+u_0)$ or $p > 2k \ \text{and} \ p^2 - p \geq a,$ then $\mathcal{L}_n^{(a)}(x)$ has no factor of degree $k$.
    \item[(ii)]If $0 \le a \le 5k$, then $\mathcal{L}_n^{(a)}(x)$ has no factor of degree $k$ except when $k=1$ and $(n, a)\in \{(2,2),(4,5)\}$.
    \end{itemize}
    \end{lemma}
    \section{Proof of Theorem \ref{n=2,3}}
    \vspace{-0.3cm}
    
    $(i)$ Consider $\mathcal L_2^{(a)}(x)=x^2-2(a+2)x+(a+2)(a+1)$. Discriminant of $\mathcal L_2^{(a)}(x)$ is $4(a+2)^2-4(a+2)(a+1)=4(a+2)$. Hence, $\mathcal L_2^{(a)}(x)$ has a root if and only if $a+2$ is a perfect square.
    
    $(ii)$ We observe that  $\mathcal L_3^{(a)}(x)=x^3-3(a+3)x^2+3(a+3)(a+2)x-(a+1)(a+2)(a+3)$ is reducible iff it has an integer root. Assume that $\xi\in\mathbb Z$ is a root of $\mathcal{L}_3^{(a)}(x)$. Further $p|\xi$ for every prime 
    $p|(a+3)$. Let $p>2$ be such that $p|(a+3)$. Let  $p^\beta\parallel(a+3)$ and $\nu_p(a_2)=\nu_p(a_1)=\beta+\delta$, where $\delta=1, 0$ according as $p=3, p\neq 3$, respectively. Then the vertices of $NP_p\left(\mathcal L_3^{(a)}(x)\right)$ are $\{(0,0), (1,\beta+\delta), (2, \beta+\delta), (3, \beta)\}$. From $\frac{\beta}{3}<\frac{\beta+\delta}{2}<\beta+\delta$, the Newton polygon has only one edge with slope $\frac{\beta}{3}$. Then the reducibility of $\mathcal{L}_3^{(a)}(x)$ implies that $3|\beta$.  Suppose $2^\lambda\parallel(a+3)$ with $\lambda>0$. Then the vertices of $NP_2\left(\mathcal L_3^{(a)}(x)\right)$ are $\{(0,0), (1, \lambda), (2, \lambda), (3, \lambda+\nu_2(a+1))\}$. If $\lambda\geq2$, then $\nu_2(a+1)=1$ and this implies the Newton polygon has only one edge with slope $\frac{\lambda+1}{3}$. Then the reducibility of $\mathcal{L}_3^{(a)}(x)$ implies that $3|(1+\lambda)$. Hence either $\lambda=1$ or $3|(1+\lambda)$. Thus, if $a$ is even, then $a+3=m^3$ for some odd $m\in\mathbb Z^+$ and if $a$ is odd, then $a+3=2\cdot m^3$ or $a+3=2^2\cdot m^3$ for some $m\in\mathbb Z^+$. 
 
\section{Main Lemmas}
    Let $n, a\in\mathbb Z^+$, $1\leq k\leq\frac n2$ and let $p$ be a prime dividing $\Delta(n - k + 1, k)$. Then 
    \begin{enumerate}[$(i)$]
    \item  there exists an integer $l$ with $0\leq l\leq k-1$ and  $u_0\geq 1$ such that $p^{u_0}\parallel (n - l)$, 
    \item there exist unique integers $j_0$ and $v_0$ with $1\leq j_0\leq p$ and $v_0\geq 1$ such that $p^{v_0}\parallel(a+j_0)$, 
    \item writing $a + j_0 = p^{v_0}t_0$ where $p\nmid t_0$, let $t_1$ be the unique integer with $1 \leq t_1 < p$ such that  $t_0 + t_1 \equiv 0 \pmod{p}$, 
    \item 
    let $c_0$ be the least positive integer satisfying 
    \begin{align}\label{c0 def}
    a+p+p^2\leq (p-1)p^{c_0},
    \end{align}
    \item for $NP_p\left(\mathcal{L}_n^{(a)}(x)\right)$, \(r_i - r_{i-1}\) represents the length of the \(i\)-th edge from the right (with \(r_0 = 0\)) along the $x$-axis and \(\phi^{(i)}\) denotes its slope.
    \end{enumerate}
    Throughout the remainder of this paper, whenever a prime $p \mid \Delta(n-k+1,k)$ is fixed, we shall use notations 
    $ j_0,  l,  u_0,  v_0, c_0, t_0, t_1, r_i \ \text{and} \ \phi^{(i)}$ as above. The dependence of these quantities on \(p\) and on the fixed parameters \(n,   a,k\) will be omitted from the notation whenever it is clear from the context.

    By Lemmas \ref{fil1} and \ref{ndivaj}$(ii)$, if $p$ is a prime divisor of  $\Delta(n-k+1, k)\cdot\Delta(n+a-k+1,k)$, and the slope $\phi^{(1)}$ of the right most edge of $NP_p\left(\mathcal{L}_n^{(a)}(x)\right)$ satisfies $\phi^{(1)}<\frac{1}{k}$, then $\mathcal{L}_n^{(a)}(x)$ has no factor of degree $k$. Therefore, if $\mathcal{L}_n^{(a)}(x)$ has a factor of degree $k$, then for every prime $p$ dividing $\Delta(n-k+1, k)\cdot\Delta(n+a-k+1,k)$, we must have $\phi^{(1)}\geq \frac{1}{k}$. Consequently, using \eqref{phix1} we obtain
           \begin{align}\label{phij>j/k}
           p\mid(\Delta(n-k+1,k)\cdot\Delta(n+a-k+1,k))\implies 
    \begin{cases}\nu_p(\Delta(a+1,r_1))-\nu_p\left(\binom{n}{r_1}\right)\geq \frac{r_1}{k},\\
        \nu_p(\Delta(a+1,r_1)\geq\frac{r_1}{k},
        \end{cases}
        \end{align}
        where, as before, $r_1$ is the length of the right-most edge of $NP_p\left(\ \mathcal{L}_n^{(a)}(x)\right)$.   
    
    In this section, we improve this criterion by showing that even when $\phi^{(1)}\geq \frac{1}{k}$, one can still rule out 
    a factor of $\mathcal{L}_n^{(a)}(x)$ of degree $k$ by using information about $r_1$ and $\phi^{(2)}$, the slope of the second edge from the right. We start with the following result, which is also an improvement of Lemma \ref{previous}$(i)$. 

   \begin{lemma}\label{slope1}
        Let $n, a\in\mathbb Z^+$, $1\leq k\leq\frac n2$ and let $p$ be a prime and $c\geq 2$ be a real number such that 
        \begin{align}\label{ck}
        p\mid\Delta(n-k+1,k), \quad p\geq 1+ck \quad {\rm and} \quad \frac{c-1}{p-1}(p^{v_0}+p)\geq c_0.
        \end{align}
        Suppose $\mathcal{L}_n^{(a)}(x)$ has a factor of degree $k$. Then 
        \begin{enumerate}[$(i)$]
            \item  $j_0\leq r_1< \max\{j_0+t_1p^{v_0},p^{v_0}+p\}$.
            \item Further if $(p-1)\nu_p(t_0+1)\leq p^{v_0}(c-1),$ then exactly one of the following holds:
            \begin{enumerate}
                \item $r_1=j_0$ and $\phi^{(2)}<\frac{1}{k}$,
                \item $r_1=p$ and $\phi^{(2)}<\frac{1}{k}$ where $l+1\leq j_0<p\leq j_0+k$,
                \item $r_1=j_0$, $r_2=p$, $\phi^{(2)}\geq \frac{1}{k}$ and $\phi^{(3)}<\frac{1}{k}$ where $l+1\leq j_0<p\leq j_0+k$. 
            \end{enumerate}
        \end{enumerate}
    \end{lemma} 
    \begin{proof}$(i).$
    Suppose $\mathcal{L}_n^{(a)}(x)$ has a factor of degree $k$. Let $p$  be a prime and $c\geq 2$ be a real number satisfying 
    \eqref{ck}. We consider $NP_p\left(\mathcal{L}_n^{(a)}(x)\right)$. From the definition of $j_0$, we observe that  $p\nmid\Delta(a+1,j)$ for $j<j_0$ and hence $\phi_j(p)\leq 0$ for $j<j_0$. Therefore $r_1\geq j_0$.  We prove that \begin{align}\label{equationj}
    \phi_j(p)<\frac{1}{k}
    \quad\text{for all}\quad
    j\geq \max\{j_0+t_1p^{v_0},\,p^{v_0}+p\}.
    \end{align} 
    Suppose, on the contrary, that there exists an integer $j$ satisfying $j\geq \max\{j_0+t_1p^{v_0},\,p^{v_0}+p\}$ such that  $\phi_j(p)\geq \frac{1}{k}.$
    Since $j\geq \max\{j_0+t_1p^{v_0},\,p^{v_0}+p\}$, we have   $(a+j_0+t_1p^{v_0})\mid\Delta(a+1, j)$. Moreover, since $a+j_0=p^{v_0}t_0 $ and $t_0+t_1\equiv 0\pmod p,$ we get $a+j_0+t_1p^{v_0}=p^{v_0}(t_0+t_1),$ and therefore $\nu_p(a+j_0+t_1p^{v_0})>v_0.$ Write $j=j_0+i_1$ where $i_1 \ge t_1p^{v_0}$. Choose $i_p$ with $1 \leq i_p \leq i_1,$ such that $$v:=\nu_p(a+j_0+i_{p})=\max_{0 \le i \le i_1}\{\nu_p(a+j_0+i)\}.$$  Thus, $v>v_0\geq 1$, and hence $v\geq 2$. Since $p^v \mid (a+j_0+i_p),$ we have $j =j_0+i_1\geq j_0+i_p\ge p^v-a$. Using \eqref{c0 def}, we obtain $j\ge p^v+p+p^2-(p-1)p^{c_0}.$ We apply the combinatorial argument of Erd\H{o}s to bound the power of $p$ in a product of consecutive integers to get $\nu_p(\Delta(a+1, j) =\nu_p(\Delta(a+j_0, i_1+1) \leq v+ \nu_p(i_1!).$ Thus, $
  \frac{1}{k} \leq   \phi_j(p)
    \leq  \frac{v+\nu_p(i_1!)}{j} < \frac{v}{j}+\frac{\nu_p(i_1!)}{i_1}<\frac{v}{j}+\frac{1}{p-1}.$
     \noindent  This together with \eqref{ck} gives 
      \begin{align}\label{c-1v}
      v>\left(\frac{1}{k}-\frac{1}{p-1}\right)j\geq 
      \left(\frac{c-1}{p-1}\right)j.
      \end{align} 
      Since $j\geq p^{v_0}+p$, this implies
    $v > \frac{c-1}{p-1}(p^{v_0}+p)\ge c_0$. Furthermore, since
    $c\geq 2$, $v\geq c_0+1$ and $j\ge p^v+p+p^2-(p-1)p^{c_0}$, it follows from \eqref{c-1v} that \begin{equation*}
        v > \frac{p^v+p+p^2-(p-1)p^{c_0}}{p-1}\geq \frac{p^v+p+p^2-(p-1)p^{v-1}}{p-1}\ge \frac{p^{v-1}+p^2+p}{p-1},
    \end{equation*} 
    which is impossible, since $p \ge 3$ and $v \ge 2$. Thus \eqref{equationj} holds. By Lemma \ref{fil1}, we have $\phi^{(1)}=\phi_{r_1}(p)\geq \frac{1}{k}$. 
    Therefore \eqref{equationj} implies that $r_1$ cannot satisfy $r_1\geq \max\{j_0+t_1p^{v_0},\,p^{v_0}+p\}$. Hence, $r_1 < \max\{j_0+t_1p^{v_0},\,p^{v_0}+p\}.$ Since $j_0 \leq r_1$,  we obtain $j_0\leq r_1<\max\{j_0+t_1p^{v_0}, p^{v_0}+p\}$. This proves $(i)$. 

\noindent
$(ii)$. We begin with the following claim.
    \begin{claim}\label{psij} 
    $\phi_{(j, j_0)}(p)<\frac{1}{k}$ for $j_0<j<\max\{j_0+t_1p^{v_0}, p^{v_0}+p\}$ except when $l+1\leq j_0<p\leq j_0+k$ and $j=p$.
    \end{claim}
    
   \noindent \textit{Proof of Claim \ref{psij}:} 
    Let $j=j_0+i_2$ for some 
    $i_2\geq 1$. 
     From \eqref{phijj'}, we have 
        \begin{align*}\label{n-j0}
          \phi_{(j, j_0)}(p)&=\frac{1}{i_2}\nu_p\left(\frac{\binom{n}{j_0}\Delta(a+j_0+1, i_2)}{\binom{n}{j}}\right)
          =\frac{1}{i_2}\nu_p\left(\frac{\Delta(a+j_0+1, i_2)j!}{\Delta(n-j_0-i_2+1, i_2)j_0!}\right).\numberthis
          \end{align*}
    \textbf{Case I: } $j_0\leq l.$
    Then in \eqref{n-j0}, $\nu_p\left(\binom {n}{j_0}\right)=0$ as $n-l< n-j_0+1$ and $p\nmid \Delta(n-j_0+1, j_0)$ and hence \begin{align}\label{l<=j0}
            \phi_{(j, j_0)}(p)\leq \frac{1}{i_2}\nu_p(\Delta(a+j_0+1, i_2)).
        \end{align}
        Let $j<j_0+t_1p^{v_0}$. 
        For $1 \leq i \leq i_2$, since $i_2 <t_1p^{v_0},$ we have $ i <t_1p^{v_0}< p^{v_0+1}.$ If $\nu_p(i) < v_0,$ then $\nu_p(a+j_0+i)=\nu_p(i).$ If $\nu_p(i)=v_0,$ then $i=\tilde{t}p^{v_0}$ for some integer $\tilde{t}$ with $1 \leq \tilde{t} <t_1$. Since \(t_1\) is the least positive integer satisfying \(t_0+t_1\equiv 0\pmod p\), we have \(p\nmid(t_0+\tilde{t})\). Therefore, $\nu_p(a+j_0+i)=v_0=\nu_p(i).$ Thus, 
        $\nu_p(a+j_0+i)=\nu_p(i), \ \text{for all}\ 1\leq i\leq i_2$ and
       \begin{align}\label{samevaluation}
       \nu_p(\Delta(a+j_0+1,i_2))=\nu_p(i_2!).
        \end{align}    
        Then by \eqref{l<=j0},we obtain $\phi_{(j, j_0)}(p)\leq\frac{\nu_p(i_2!)}{i_2}<\frac{1}{p-1}<\frac{1}{k}.$
        
        Let $ j \geq j_0+t_1p^{v_0}$. Since $j<\max\{j_0+t_1p^{v_0}, p^{v_0}+p\}$  this case is possible only if $ j_0+t_1p^{v_0}< p^{v_0}+p$. Hence, $\max\{j_0+t_1p^{v_0}, p^{v_0}+p\}=p^{v_0}+p,$ and therefore, $ j < p^{v_0}+p$. Thus, $j_0+t_1p^{v_0} \leq j=j_0+i_2<p^{v_0}+p$. This gives 
        $t_1=1$ and $p^{v_0}\leq i_2< p^{v_0}+(p-j_0)$. Let $\alpha=\nu_p(t_0+1)$. We have $\nu_p(a+j_0+p^{v_0})=v_0+\alpha$ since $a+j_0=p^{v_0}t_0$. Also, for \(1\leq i\leq i_2\), the maximum value of $\nu_p(a+j_0+i)$ is attained at \(i=p^{v_0}\). Therefore, $\nu_p(\Delta(a+j_0+1, i_2))\leq \nu_p(a+j_0+p^{v_0})+\nu_p((p^{v_0}-1)!)=\alpha+\nu_p(p^{v_0}!).$ Thus by \eqref{l<=j0}, Lemma \ref{sigma} and using 
          $(p-1)\alpha\leq p^{v_0}(c-1)$,  we obtain 
          \begin{align*}
            \phi_{(j, j_0)}(p)&\leq\frac{\alpha+\nu_p(p^{v_0}!)}{i_2}
            < \frac{\alpha}{p^{v_0}}+\frac{p^{v_0}}{(p-1)p^{v_0}}\leq \frac{\alpha(p-1)}{(p-1)p^{v_0}}+\frac{1}{p-1} \leq\frac{c}{p-1}\leq\frac{1}{k}.
        \end{align*}
        \textbf{Case II: } $j_0\geq l+1$. We proceed similarly to Case I. Let 
        $j\geq j_0+t_1p^{v_0}$. Then we may assume that $j_0+t_1p^{v_0}\leq j=j_0+i_2<p^{v_0}+p$ implying 
        $t_1=1, j_0<p$ and $p^{v_0}\leq i_2< p^{v_0}+p$. This imply $\nu_p(j!)=\nu_p(i_2!)=\nu_p(p^{v_0}!)$ and 
       $\nu_p(\Delta(a+j_0+1,i_2))=\nu_p(\Delta(a+j_0+1, p^{v_0}))\leq \nu_p((a+j_0+p^{v_0}))+\nu_p((p^{v_0}-1)!)\leq \nu_p(t_0+1)+v_0+\nu_p((p^{v_0}-1)!)=\alpha+\nu_p((p^{v_0})!).$ These together with \eqref{n-j0} and $ \nu_p(\Delta(n-j_0-i_2+1, i_2))\ge \nu_p(i_2!)$ gives 
    \begin{align*}
          \phi_{(j, j_0)}(p)&\leq \frac{1}{i_2}\left(\nu_p(\Delta(a+j_0+1, p^{v_0}))\right)
          \leq \frac{\alpha+\nu_p(p^{v_0}!)}{p^{v_0}}
          <\frac{c-1}{p-1}+\frac{1}{p-1}\leq\frac{1}{k},
      \end{align*} 
      since $\alpha(p-1) \leq p^{v_0}(c-1)$ and $p\geq 1+ck$.
        
    Let $j< j_0+t_1p^{v_0}$. In this case, we again have $i_2 < t_1p^{v_0}<p^{v_0+1}$ and hence \eqref{samevaluation} holds. Using this, along with $ \nu_p(\Delta(n-j_0-i_2+1, i_2))\ge \nu_p(i_2!)$ and applying Lemma \ref{sigma} to \eqref{n-j0}, we obtain 
     \begin{align*}\label{sigmaj0-sigmaj}
         \phi_{(j, j_0)}(p)&\le \frac{\nu_p(j!)-\nu_p(j_0!)}{i_2}
         \leq \frac{i_2+\sigma_p(j_0)-\sigma_p(j)}{(p-1)i_2}
         =\frac{1}{p-1}+\frac{\sigma_p(j_0)-\sigma_p(j)}{i_2(p-1)}.\numberthis
     \end{align*}
    Since $p-1\geq ck\geq 2k$, when $i_2\geq p-1$, $\phi_{(j,j_0)}(p)\leq \frac{\nu_p(j!)}{i_2}\leq \frac{j-\sigma_p(j)}{i_2(p-1)}<\frac{2}{p-1} <\frac{1}{k}.$ Hence, we may assume that \(\sigma_p(j_0)>\sigma_p(j)\), \(i_2\leq p-2\) and \(j_0<p\leq j\).
    Writing $j=p+(j_0+i_2-p)$, we have $\sigma_p(j)=1+(j_0+i_2-p)$ implying  $\phi_{(j, j_0)}(p)\leq \frac{1}{i_2}<\frac{1}{k} \ \text{if} \ i_2>k.$ Thus, we may also suppose that $i_2\leq k$. Thus $p\leq j=j_0+i_2\leq j_0+k$.  If $j\neq p$, then $j=p+h$ for some $1\leq h\leq j_0+k-p<k$. Clearly $\phi_{(j,j_0)}(p)\leq \phi_{(j,j-1)}(p)=\phi_{(p+h,p+h-1)}(p)$. Then  $a+j_0<a+p+h\leq a+j_0+k+1<a+j_0+p$ and \eqref{phijj'} gives  $\phi_{(j,j_0)}(p)\leq\nu_p(p+h)+\nu_p(a+p+h)=0,$ and this proves Claim \ref{psij}.

    We now complete the proof of \((ii)\). From part \((i)\), we know that $j_0\leq r_1<\max\{j_0+t_1p^{v_0},\,p^{v_0}+p\}.$ Moreover, Claim \ref{psij}, together with the inequality $\phi_j(p)\leq \phi_{(j,j_0)}(p),$ shows that $\phi_j(p)<\frac{1}{k}$ for every integer \(j\) satisfying $j_0<j<\max\{j_0+t_1p^{v_0},\,p^{v_0}+p\},$ except possibly when $j=p$ and $l+1\leq j_0<p\leq j_0+k.$ On the other hand, by Lemma \ref{fil1}, under the assumption that \(\mathcal{L}_n^{(a)}(x)\) has a factor of degree \(k\), we may assume that $\phi^{(1)}=\phi_{r_1}(p)\geq \frac{1}{k}.$ Hence either $r_1=j_0,$ or $r_1=p\quad\text{with}\quad l+1\leq j_0<p\leq j_0+k.$ First suppose that $r_1=p.$ Then necessarily $l+1\leq j_0<p\leq j_0+k$ by Claim \ref{psij}. As $r_1=p$, $r_2\neq p$ and as $\phi^{(2)}=\phi_{(r_2, p)}<\phi_{(r_2,j_0)}$, again by Claim \ref{psij}, we obtain $\phi^{(2)}<\frac{1}{k}.$ Thus alternative \((b)\) holds. It remains to consider the case $r_1=j_0.$ If $\phi^{(2)}<\frac{1}{k},$ then alternative \((a)\) holds. Hence, we may assume that $\phi^{(2)}\geq \frac{1}{k}.$ Since \(n-r_2\) denote the second vertex position from the right of $NP_p(\mathcal{L}_n^{(\alpha)}(x))$, we have $ r_2>r_1=j_0$. Moreover, $\frac{1}{k} \leq \phi^{(2)}=\phi_{(r_2, j_0)}<\phi_{r_2}$. Thus by \eqref{equationj}, we have $j_0<r_2<\max\{j_0+t_1p^{v_0},\,p^{v_0}+p\}.$ Applying Claim \ref{psij} with \(j=r_2\), and using $\phi^{(2)}=\phi_{(r_2,j_0)}(p)\geq \frac{1}{k},$ we get $r_2=p$ and $l+1\leq j_0<p\leq j_0+k.$ Finally, applying Claim \ref{psij} once more to the possible vertices following \(r_2=p\), we obtain $\phi^{(3)}<\frac{1}{k}.$ Therefore alternative \((c)\) holds. This completes the proof of \((ii)\).
\end{proof}
    \begin{note}
    \phantomsection
    \label{remark}
       \noindent 
       $(i)$ We have $t_0=\frac{a+j_0}{p^{v_0}}<\frac{(p-1)p^{c_0}-p-p^2+p}{p^{v_0}}< (p-1)p^{c_0-v_0}$ giving $t_0+1\leq (p-1)p^{c_0-v_0}<p^{c_0-v_0+1}$. Hence  $\alpha=\nu_p(t_0+1)\leq c_0-v_0.$
       
        \noindent 
       $(ii)$ If $r_1=j_0$  and $j_0\leq l$, then by Lemma \ref{lucas}, 
       $\phi^{(1)}=\frac{v_0}{j_0}$.\\         
        \noindent 
       $(iii)$ If $r_1=j_0$, $j_0\neq p$ and $j_0\geq l+1$, then $    \nu_p\left(\binom{n}{j_0}\right)=u_0.$ Thus,   $\phi^{(1)}=
        \begin{cases}
            \frac{v_0-u_0}{j_0},&\text{if } j_0\neq p,\\
            \frac{v_0-u_0+1}{j_0},&\text{if } j_0=p.
        \end{cases}$
    \end{note}

    \begin{lemma}\label{v0gequ0+1}
        Let $50<a\leq 1000$, $1\leq k\leq\frac n2$ and suppose that $\mathcal L_n^{(a)}(x)$ has a factor of degree $k$. If 
        $p$ is a prime satisfying \eqref{ck} and $(p-1)\nu_p(t_0+1)\leq p^{v_0}(c-1)$, then
        $j_0\geq l+1$ and  $v_0\geq u_0+1$. 
        
        Further, $r_1=j_0$, except for the case  $k=1$, $p=3$, $j_0=2$ and $v_0-u_0=2$. In this exceptional case, $r_1=p$. If $r_2\neq p$, and if $\phi^{(1)}=\frac{r}{s}$, $(r,s)=1$, then $k\in\{\beta s, \beta s+1\ldots, \beta s+l\}$ for some integer $\beta$ satisfying $1\leq \beta\leq \frac{r_1}{s}$.
        \end{lemma}
    \begin{proof}
    Suppose, on the contrary, that  $j_0\leq l$. Then, by Lemma \ref{slope1}$(ii)$, only alternative $(a)$ can occur. Hence  $r_1=j_0$ and $\phi^{(2)}<\frac{1}{k}$.
    By Lemma \ref{cor-dumas}, the degree $k$ can be written as $k=\sum_{i=1}^{t}\epsilon_ib_i,$ where $b_i=\Tilde{x}_i-\Tilde{x}_{i-1},\epsilon_i\in\{0,1\}$ and $\Tilde{x_i}$ are the $x$-coordinates of  lattice points in $NP\left(\mathcal L_n^{(a)}(x)\right)$ for $1 \le i \le t$. Since $p\mid(n-l)$,  and $p\mid(a+j_0)$, we have $p\mid(n+a)-(l-j_0)$. Thus, $\nu_p((n+a)\cdots(n+a-(l-j_0))=\nu_p(n+a-l+j_0)$, since $l-j_0<p$. Therefore $
    \nu_p(a_{n-j})=0$ for $j\leq l-j_0$ and $\nu_p(a_{n-j}) >0$ for $j=l-j_0+1.$ Thus, the length of the edge with zero slope in $NP_p\left(\mathcal L_n^{(a)}(x)\right)$ is $l-j_0$. We observe that the distance between $x$- coordinates of any two lattice points on an edge with positive slope $<\frac{1}{k}$ is greater than $k$. Since we also have $\phi^{(2)} < \frac{1}{k}$, it follows from the expression of $k$, $\epsilon_i\neq 0$ only possibly for those $b_i'$s coming from the rightmost edge or from the edge with zero slope. Using the fact that the length of the edge with zero slope is $l-j_0,$ and that the length of the right-most edge is $r_1=j_0$, we obtain $k\leq l-j_0+j_0=l<k$, which is a contradiction. Hence, we can conclude that $j_0\geq l+1$. If $v_0\leq u_0$, then by Remark \ref{remark} $(iii)$ and $(iv)$,  $\phi^{(1)}=\phi_{r_1}(p)\leq 0$ if $r_1=j_0$ and $\phi^{(1)}\leq \frac{1}{p}$ if $r_1=p$.    
    Hence in either case $\phi^{(1)}<\frac{1}{k}$ which contradicts Lemma \ref{fil1}. Therefore $v_0\geq u_0+1$.  

    By Lemma \ref{slope1}, $r_1=j_0$ or $r_1=p$. Suppose that $r_1=p$, $p\neq j_0$. Then by Lemma \ref{slope1}$(ii)$,  $\phi^{(2)}<\frac{1}{k}$. Here $\phi^{(1)}=\frac{v_0-u_0+1}{p}$. 
    Assume that $(v_0-u_0+1,p)=1$. Then the only lattice points in the right-most edge are the endpoints, so that the distance between $ x$-coordinates of the lattice points of the rightmost edge is $p>k$. Also, the distance between $ x$-coordinates of any two lattice points on an edge with positive slope $<\frac{1}{k}$ is greater than $k$. Therefore, in the expression of $k$, we have $\epsilon_i\neq 0$ only possibly for those $b_i'$s coming from the edge with zero slope, so that $k$ is at most equal to the length of the edge of zero slope. 
    From $p\mid(n-l)$, we have $\nu_p(a_{n-j})>0$ for $j>l$. Hence, the length of the edge with zero slope of $NP_p\left(\mathcal L_n^{(a)}(x)\right)$ is at most $l$. Since $l<k$, this gives a contradiction, and we may assume that $(v_0-u_0+1,p)>1$. For \(p\geq 5\), since $p^{v_0}\leq a+j_0\leq 1000+p<p^5,$ we have \(v_0\leq 4\). Together with \(u_0\geq 1\), this gives $v_0-u_0+1\leq 4<p,$ contradicting $(v_0-u_0+1,p)>1$. Therefore $p=3$ and $j_0\leq 2$. Since \(p\geq 1+ck\) and \(c\geq 2\), we get \(k=1\). Also \(v_0\leq 6\), and $(v_0-u_0+1,3)>1$ implies $v_0-u_0\in\{2,5\}.$ If $v_0-u_0=5$, then $\phi_p(p)=2$ and $\phi_{j_0}(p)\in\left\{5,\frac{5}{2}\right\}.$ On the other hand, if $v_0-u_0=2$ and $j_0=1$, then $\phi_p(p)=1\quad\text{and}\quad\phi_{j_0}(p)=2.$ Thus, in either situation, we have $\phi_{j_0}(p)>\phi_p(p),$ which implies $r_1=j_0$. The exceptional case occurs when $v_0-u_0=2$, $j_0=2$, $\phi_p(p)=1$ and $\phi_{j_0}(p)=1$, implying $r_1=p$.
 
     Finally, suppose that \(r_2\neq p\). Then by Lemma \ref{slope1}\((ii)\), we have $r_1=j_0\ \text{and}\ \phi^{(2)}<\frac{1}{k}.$ Let $\phi^{(1)} =  \frac{r}{s}$, the distance between any two lattice points on the right-most edge is a multiple of $s$. Hence by Lemma \ref{cor-dumas}, $k = \sum_{i=1}^{l} \epsilon_i + \beta s$, where $\beta s \leq r_1$.
    \end{proof}

    \begin{lemma}\label{x2neqp}
        Let $50<a\leq 1000$, $1\leq k\leq\frac n2$ and suppose that $\mathcal{L}_n^{(a)}(x)$ has a factor of degree $k$. If  $p\geq\max\{1+2k, 11\}$ is a prime divisor of $\Delta(n-k+1,k)$, then  
    $r_1=j_0, r_2\neq p, j_0\geq l+1, v_0=2, u_0=1,$ and $k\in\{j_0,j_0+1,\ldots,j_0+l\}.$
    \end{lemma}
    \begin{proof}
        Since $1000+p+p^2<(p-1)p^2$ for $p\geq 11$ and since
        $50<a\leq 1000$, we have $c_0=2$ by \eqref{c0 def}. Further, by Remark \ref{remark}$(i)$, $0\leq \nu_p(t_0+1)\leq 2-v_0\leq 1$. Hence 
        $v_0\leq 2$, and the conditions of Lemma \ref{v0gequ0+1} are satisfied with $c=2$. Therefore, by Lemma \ref{v0gequ0+1}, we obtain $j_0\geq l+1$, $v_0\geq u_0+1$ and $r_1=j_0$. Since $u_0\geq 1$, it follows that  $v_0=2$ and $u_0=1$. Suppose now that $r_2=p$. Then, by Lemma \ref{slope1}$(ii)$, we have  $l+1\leq j_0<p\leq j_0+k$. In particular, $j_0\geq p-k$. Since $r_1=j_0$, the rightmost edge has slope $
    \phi^{(1)}=\frac{v_0-u_0}{j_0}=\frac{1}{j_0}<\frac{1}{k}$ since $j_0\geq p-k$ and $p\geq 1+2k$. 
    Hence $r_2\neq p$. Further from Lemma \ref{v0gequ0+1} and $\phi^{(1)}=\frac{v_0-u_0}{j_0}=\frac{1}{j_0}$,  we obtain $k\in\{j_0, j_0+1,\ldots, j_0+l\}$.
    \end{proof}

\section{\texorpdfstring{Proof of Theorem \ref{main} with $k\geq 2$}{Proof of Theorem \ref{main} with k>= 2}} 
    By Lemma \ref{previous}$(ii)$, we may assume that $\max\{5k, 50\}<a\leq 1000$.  In particular, $k<200$. Thus, in this section, we assume that $n\geq 2k,\quad
        k<200,\quad
        \max\{5k, 50\}< a\leq 1000.$ 

    \begin{lemma}\label{pmax1+2k}
    Let $k\geq 1$, $n\geq 2k$ and $50<a\leq 1000$. Let $p\geq \max\{1+2k, 37\}$ be a prime such that $p \mid \Delta(n-k+1,k)$. Then $\mathcal{L}_n^{(a)}(x)$ has no factor of degree $k$.
    \end{lemma}
    \begin{proof}
    By Lemma \ref{x2neqp}, $v_0\geq u_0+1$. As $u_0\geq 1$, $p^2 \mid \Delta(a+1, p)$. That is, $a+p\geq p^2$, which is a contradiction since $p\geq 37$. Hence, the proof.
    \end{proof}
    \begin{lemma}\label{nofactor6}
       Let $k\geq 6$, $n\geq 2k$ and $50<a\leq 1000$. Then $\mathcal{L}_n^{(a)}(x)$ has no factor of degree $k$.
    \end{lemma}
    \begin{proof}
        Assume that $\mathcal{L}_n^{(a)}(x)$ has a factor of degree $k$. We check that $\max\left\{2k+13, \frac{541}{262}k\right\}=2k+13$, since $k<200$. Therefore by using Lemma \ref{2k+13}, we have $\text{if } n\geq  2k+13, \text{ then }P(\Delta(n-k+1,k))\geq 2k+1.$  Let $k\geq 18$. Then $\max\{1+2k, 37\}=1+2k$ and by Lemma \ref{pmax1+2k} we conclude that $2k\leq n\leq 2k+12$. Further,  $n+a-k+1>2k+5k-k+1= 6k+1$. Thus by Lemma \ref{4.42k}, we have $P(\Delta(n+a-k+1,k))> 4.42k.$ Let $p=P(\Delta(n+a-k+1,k)$. Since $k \geq 18$ and $p>4.42k > 37$, we have $p^2-p\geq 1000\geq a$. Therefore $p$ must divide $\Delta(a+1,k)$ by Lemma \ref{previous}$(i)$. Let $0\leq i\leq k-1$ and $1\leq j\leq k$ be such that $p \mid (n+a-i)$ and $p \mid (a+j)$. Then $p \mid n-(i+j)$, $1\leq i+j< 2k$. That is, $p\leq n-1\leq 2k+11$, which is a contradiction. Let $6\leq k<18$. Then $\max\{1+2k, 37\}=37$. Using Lemma \ref{table-najman}, 
        we have $P(\Delta(n-k+1,k))\geq37$ for $n-k+1> 31$ and hence $n\geq 37$. Therefore, by Lemma \ref{pmax1+2k}, we can assume that $n\leq 36$. 
        Then since $n+a-k+1>6k+1\geq 37$, by Lemma  \ref{table-najman},  $P(\Delta(n+a-k+1, k))>37$. That is, there exists a $p> 37$ such that  $p \mid (n+a-i)$, where $0\leq i\leq k-1$. Hence, by Lemma \ref{previous}$(i)$, as $p^2-p>1000$ and $p>2k$, we have $p \mid (a+j)$ for some $1\leq j\leq k$. That is, $p \mid n-(i+j)$ for some $1\leq i+j<2k$ and $p\leq n-1\leq 35$, which is a contradiction. 
    \end{proof}

    We now assume that \(1\leq k\leq 5\). Using Lemmas \ref{slope1}, \ref{v0gequ0+1} and \ref{x2neqp}, we derive restrictions on the prime divisors of \(\Delta(n-k+1,k)\). 
       
    \begin{lemma}\label{lemmak=1to5}
        Let $1\le k\le 5$, $n\ge 2k$, and $50<a\le 1000$. Assume $\mathcal{L}_n^{(a)}(x)$ has a factor of degree $k$. If $p$ is  a prime dividing $ \Delta(n-k+1,k)$, then $p\le 31$. Further, for $11\le p\le 31$, we have $\nu_p\bigl(\Delta(n-k+1,k)\bigr)=1,$
    and 
    $p^2\parallel \Delta\left(a+k-\left\lceil\frac{k}{2}\right\rceil+1,\left\lceil\frac{k}{2}\right\rceil\right).$
    
    Consequently, $
    \omega_{11}\!\bigl(\Delta(n-k+1,k)\bigr)\le 1,$ except in the exceptional case $k=5$ with $a\in\{358,842\}$, where $\omega_{11}\!\bigl(\Delta(n-k+1,k)\bigr)\le 2$. Here $\omega_p(m)$ is defined as in \eqref{omega}.  Furthermore, in the case $k=2$, we also have $\omega_7\bigl(\Delta(n-k+1,k)\bigr)\le 1,$
    and if $p=7$, then either $p^2\mid(a+2)$ or $p^3\parallel(a+4)$, with $p\parallel n(n-1)$.
    \end{lemma}
    \begin{proof}
     Assume that $\mathcal{L}_n^{(a)}(x)$ has a factor of degree $k$, where $1\leq k\leq 5$. Then by Lemma \ref{pmax1+2k},  we have $p\leq 31$. Let $11 \le p \le 31$. By Lemma \ref{x2neqp}, we have $l+1\leq j_0\leq k\leq j_0+l$, $v_0=2$ and $u_0=1$. Since $u_0=1$ and \(p>k\), it follows that $p \parallel \Delta(n-k+1, k)$.    
        \begin{center}\setlength{\tabcolsep}{6pt}
        \scalebox{0.7}{
        \begin{tabular}{|c|c||cc||ccc||cccc||ccccc|}
        \hline
        \textbf{$k$}   & 1 & \multicolumn{2}{c||}{2} & \multicolumn{3}{c||}{3} & \multicolumn{4}{c||}{4} & \multicolumn{5}{c|}{5} \\
        \hline
        \textbf{$l$}   & 0 & 0 & 1 & 0 & 1 & 2 & 0 & 1 & 2 & 3 & 0 & 1 & 2 & 3 & 4 \\
        \hline
        \textbf{$j_0$} & 1 & 2 & 2 & 3 & 2,3 & 3 & 4 & 3,4 & 3,4 & 4 & 5 & 4,5 & 3,4,5 & 4,5 & 5 \\
        \hline
        \end{tabular}}
        \end{center}
       \noindent Further using $p^2 \parallel(a+j_0)$, together with the possible values of \(j_0\) listed in the table, we obtain  $p^2 \parallel(a+k)$ if $k\in\{1,2\}$, $p^2 \parallel(a+k-1)(a+k)$ if $k\in\{3, 4\}$ and $p^2 \parallel(a+k-2)(a+k-1)(a+k)$ if $k=5$.
     If $\omega_{11}(\Delta(n-k+1,k))\geq 2$, then we have distinct primes $p,q$ such that $p^2\mid (a+k-i)$ and $q^2\mid (a+k-j)$, where $0\leq i-j< \left\lceil\frac{k}{2}\right\rceil$. Then by applying Lemma \ref{equationmm'} with $N=a+k-j$, we get the only possible case is when $k=5$,  $p^2 \mid (a+3)$ and $q^2 \mid (a+5)$. Thus, we obtain $(a, p,q, m, m')\in\{(842, 11, 13, 7, 5),(358, 11, 19, 3, 1)\}$. Therefore $\omega_{11}(\Delta(n-k+1,k))\leq 1$ except possibly when $k=5$ and $a\in\{358, 842\}$.  For $a\in\{358, 842\}$, as above, we have $\omega_{11}(\Delta(n-k+1, k))\leq 2$.
    
    Next, we consider the case when $k=2$ and $p=7$. Then in Lemma \ref{slope1}, we take $c=3$ and $c_0=3$. Hence $\frac{c-1}{p-1}(p^{v_0}+p)\geq \frac{4p}{p-1}\geq c_0$. 
     Let $\alpha=\nu_p(t_0+1)$. Suppose that $p^{v_0}(c-1)<\alpha(p-1)$. From Remark \ref{remark}$(i)$, we have $\alpha\leq c_0-v_0$. Therefore $\frac{2p^{v_0}}{p-1}=\frac{p^{v_0}(c-1)}{p-1}< \alpha \leq 3-v_0$, which is a contradiction. Hence $\alpha(p-1)\leq p^{v_0}(c-1)$. Thus, by Lemma \ref{v0gequ0+1} and as $a\leq 1000$, we obtain $j_0\geq l+1$, $r_1=j_0$ and $1\leq v_0-u_0\leq 2$.  Suppose $r_2=p$. Then, by Lemma \ref{slope1}, we have $1\leq j_0\leq 6$, $p\leq j_0+k$.  
    Since $p=7$ and $k=2$, this gives $j_0=5$ or $j_0=6$. Hence $\frac{v_0-u_0}{j_0}\leq \frac{2}{5}.$ On the other hand, by Remark \ref{remark}$(iii)$, $\phi^{(1)}=\phi_{j_0}(p)=\frac{v_0-u_0}{j_0}>\frac{1}{2}$ which is a contradiction. Thus $r_2\neq p$. If $v_0-u_0=1$, then $(v_0-u_0, j_0)=1$, and from $j_0\leq k\leq j_0+l$ with $k=2$, we get the following. If $l=0$, then $j_0=k=2$; if $l=1$, then since $j_0\geq l+1=2$, again $j_0=2$. Hence $j_0=2$, and therefore $p^2\mid(a+2)$. Now consider $v_0-u_0=2$. If $j_0\in\{1, 3, 5\}$, then $(v_0-u_0,j_0)=1$,  and hence $j_0\leq k\leq j_0+l$, which is impossible since $j_0\geq l+1$ and $k=2$. Thus $j_0\in\{2, 4, 6\}$. But $k(v_0-u_0)\geq j_0$, and hence $j_0=2$ or $j_0=4$. Since $v_0\geq 3$, we have $p^3\mid (a+2)(a+4)$. But  $a\leq 1000$ will imply $p^3\parallel(a+2)(a+4)$ and hence $u_0=1$ implying $p\parallel n(n-1)$. Therefore $p^2\mid (a+2)$ or $p^3\parallel (a+2)(a+4).$  Let $a\notin\{339, 682\}$. Then, if $7\mid n(n-1)$, as $7^3\nmid (a+4)$, we have $7^2\mid (a+2)$.
  Suppose that $\omega_7(n(n-1))\geq 2$. Then there exist two distinct primes $p,q\geq 7$ such that $p^2q^2\mid(a+2)$, which is impossible since $a+2\leq 1002$. Hence $\omega_7(n(n-1))\leq 1.$ For $a\in\{339, 682\}$, suppose there exist a prime $p>7$ such that $p\mid n(n-1)$. Then $p^2\mid(a+2)$, which is impossible  and hence $\omega_7(n(n-1))\leq 1$.
    \end{proof}
    We next derive analogous conditions for prime divisors of $\Delta(n+a-k+1,k)$ when $1\leq k\leq 5$.
    \begin{lemma}\label{n+a- 1<=k<=5}
    Let $1\le k\le 5$, $n\ge 2k$, and $50<a\le 1000$. Assume that $\mathcal{L}_n^{(a)}(x)$ has a factor of degree $k$.  Let $p$ be a prime such that $p \mid \Delta(n+a-k+1,k)$.
    Then the following statements hold:    
    \begin{enumerate}[$(i)$]
    \item If $p\ge 37$, then
    $p \parallel \Delta(a+1,k)
    \quad \text{and} \quad
    p \mid \Delta(n-2k+1,k)$.
    Moreover, $
    P\!\left(\Delta(n+a-k+1,k)\right)
    \le
    P\!\left(\Delta(n-2k+1,k)\right).$
    \item Suppose \(k=1\). Then \(r_1<12\) for \(p=3\). For \(p=5\) and \(p=7\), we have $
p^j\mid(a+j)$ for some \(1\leq j\leq 4\) and \(1\leq j\leq 3\), respectively. If \(11\leq p\leq 31\), then either $
p\mid(a+1)
\ \text{or}\ 
p^2\mid(a+2).$
    \item Suppose $k=2$. If $11\le p\le 31$, then either $p \mid (a+1)(a+2)
    \quad \text{or} \quad
    p^2 \mid (a+3)(a+4)$. Moreover, $p \mid \Delta(n-5,5),$ and hence $P\!\left(\Delta(n+a-1,2)\right)
    \le
    P\!\left(\Delta(n-5,5)\right)$.
    \end{enumerate}
    \end{lemma}
    \begin{proof}
    Assume that $\mathcal{L}_n^{(a)}(x)$ has a factor of degree $k$. Let $c\ge 2$ be the greatest integer satisfying $p \ge 1+ck$. Consider $NP_p\left(\mathcal{L}_n^{(a)}(x)\right)$.  
    As $\nu_p(\Delta(a+1, j))=0$ for $j<j_0$, we have $r_1\geq j_0$ by \eqref{phij>j/k}. Following the approach of Lemma \ref{slope1}, we derive an upper bound for $r_1$. We write $r_1=j_0+i_1$ and define $\Tilde{v}=\max\{\nu_p(a+j_0+i):0\leq i\leq i_1 \}.$
     Let $r_1>j_0$. Then $i_1\geq 1$ and  $\Tilde{v}\geq v_0$. Moreover, $r_1\geq p^{\Tilde{v}}-a$. Further, $
         \frac{1}{k}\leq \phi_{r_1}(p)\leq \frac{\nu_p(\Delta(a+j_0, i_1+1))}{r_1}
         \leq \frac{\Tilde{v}}{r_1}+\frac{\nu_p(i_1!)}{i_1}<\frac{\Tilde{v}}{r_1}+\frac{1}{p-1},$ by Lemma \ref{sigma}.
        Therefore we have $\frac{\Tilde{v}}{r_1}> \frac{1}{k}-\frac{1}{p-1}\geq \frac{c-1}{p-1}$. Let $r_1=j_0$. Then $\Tilde{v}=v_0$ and  $\frac{1}{k}\leq\phi_{r_1}(p)\leq \frac{v_0}{r_1}=\frac{\Tilde{v}}{r_1}.$ 
        Hence, in both cases, $r_1< \frac{p-1}{c-1}\Tilde{v}.$ Consequently, $p^{\Tilde{v}}-a\leq r_1\leq  (p-1)\Tilde{v}$. Since $a\leq 1000$, this yields the following bounds on  $\Tilde{v}$, and hence the corresponding bounds on $r_1$:
    \begin{center}\scalebox{0.7}{
            \begin{tabular}{|c|c|c|c|c|c|}
            \hline
            $p$ & $\geq 37$ & $\geq 11$ & $7$ &$5$ &$3$\\
            \hline
              $\Tilde{v}$ & $1$ & $\leq 2$& $\leq 3$& $\leq 4$& $\leq 6$\\
              \hline
              $r_1 < $& $\frac{p-1}{c-1}$  & $\frac{2(p-1)}{c-1}$& $ \frac{18}{c-1}$& $\frac{16}{c-1}$& $ 12$\\
            \hline
            \end{tabular}}
       \end{center}
        Note that $k=1$ for $p=3$ and hence $c=2$.      

        \noindent $(i).$ Let $p\mid (n+a-i)$ for some $0\leq i\leq k-1$. If $p\geq 37 $, then $r_1\leq \frac{p-1}{c-1}<p$. Further, $\nu_p(\Delta(a+1, p))=\nu_p(a+j_0)=1$ since $a\leq 1000$ and $j_0\leq p$. Consider $\phi_j(p)$ for $j<p$. 
        As $\nu_p(\Delta(n-j_0+1, j_0))\leq \nu_p(\Delta(n-j+1, j))$ for $j_0<j<p$, by the definition of $\phi_j(p)$  
          we have $\phi_{j_0}(p)>\phi_j(p)$ and therefore $r_1=j_0$.
        Further, $\phi_{r_1}(p)\leq \frac{1}{r_1}$. By Lemma \ref{fil1}, we may assume that $
\phi_{r_1}(p)\geq \frac{1}{k}.$ Hence \(r_1\leq k\), and so $p\parallel\Delta(a+1,k)$.  Since $p\mid(n+a-i)$ and $p\mid(a+j_0)$, we have $p\mid(n-(j_0+i))$ where $1\leq j_0+i\leq 2k-1$. It follows that $p\mid\Delta(n-2k+1, 2k)$. On the other hand, by Lemma \ref{pmax1+2k}, we have $P(\Delta(n-k+1,k))\leq 31$. Therefore, we conclude that $p\mid \Delta(n-2k+1,k)$. Consequently, each prime $p\geq 37$ dividing $\Delta(n+a-k+1,k)$ also divides $\Delta(n-2k+1,k)$. Hence $P(\Delta(n+a-k+1,k))\leq P(\Delta(n-2k+1,k)).$\\ 
        \noindent$(ii).$ Let $k=1$. Then $c=p-1$. From the above table, we have        
        \begin{center}\scalebox{0.7}{
            \begin{tabular}{|c|c|c|c|c|}
            \hline
              $p$   & $3$&$5$ &$7$&$\geq 11$ \\
              \hline
               $r_1\leq$ &$11$ &$5$ &$3$ &$2$\\
               \hline
            \end{tabular}}
        \end{center} 
        Let $p=5$. By \eqref{phij>j/k}, we have $\nu_5(\Delta(a+1,r_1))\geq r_1$. Since $r_1\leq p=5$ and $a\leq 1000$, $\nu_5(\Delta(a+1,r_1))\leq 4$. Therefore there exists $j$, $1\leq j\leq 4$, such that $\nu_5(a+j)\geq j.$ Equivalently, $5^j\mid(a+j)$ for some $1\leq j\leq 4$. Let $p=7$. Then $r_1\leq 3<p$. Arguing as above, there exists $j$, $1\leq j\leq 3$, such that $\nu_7(a+j)\geq j$.  Equivalently, $7^j\mid(a+j)$ for some $1\leq j\leq 3$. Similarly, let $p\geq 11$. Then   $\nu_p(a+1)\geq 1$ or $\nu_p(a+2)\geq 2$.  \\
        \noindent $(iii)$. Let $k=2$.
        If $p\geq 11$, then $r_1\leq \frac{4(p-1)}{p-3}$. Thus if $p=11$, then $r_1\leq 5$ and if $p\geq 13$, then $r_1\leq 4$. Since $a\leq 1000$, it follows that $p\mid(a+1)(a+2)$ or $p^2\mid(a+3)(a+4)$. Hence $p\mid(a+j_0)$ for some $1\leq j_0\leq 4$ and $p\mid(n+a-i)$, $0\leq i\leq 1$, implying $p\mid n-(i+j_0)$, where $1\leq i+j_0\leq 5$. That is, $p\mid\Delta(n-5, 5)$ and since $\Delta(n+a-1, 2)$ has a prime divisor $p\geq 11$ by our assumption, we obtain $P(\Delta(n+a-1,2))\leq P(\Delta(n-5,5))$.
        \end{proof}

        \begin{lemma}\label{nofactor3-5}
            Let $3\leq k\leq 5$, $50<a\leq 1000$, and $n\geq 2k$. Then $\mathcal{L}_n^{(a)}(x)$ has no factor of degree $k$.
        \end{lemma}
        \begin{proof}
       Assume that $\mathcal{L}_n^{(a)}(x)$ has a factor of degree $k$, where $3 \leq k \leq 5$. By Lemma \ref{pmax1+2k}, we have    $P\bigl(\Delta(n-k+1,k)\bigr) \leq 31$.
    The complete list of positive integers $m$ satisfying $P(m(m-1)) \leq 31$ is given in \cite[Table IA]{lehmer}. Using this table, we explicitly determine all pairs $(n,k)$ with $3 \leq k \leq 5, n \geq 2k$ for which $P\bigl(\Delta(n-k+1,k)\bigr) \leq 31$. This gives $87$, $46$, and $31$ possible values of $n$ for $k=3,4,5$, respectively. Hence there are $164$ such pairs $(n,k)$ in total. These pairs must then be considered for each integer $a$ with $50<a\leq 1000$. By Lemma \ref{lemmak=1to5}, we have $\omega_{11}\bigl(\Delta(n-k+1,k)\bigr)\leq 1$ except when $(k,a)\in\{(5, 358), (5, 842)\}$. In these exceptional cases, we have $\omega_{11}\bigl(\Delta(n-k+1,k)\bigr)\leq 2$. Furthermore, if $\omega_{11}\bigl(\Delta(n-k+1,k)\Bigr)=2$, then as discussed in the proof of Lemma \ref{lemmak=1to5}, this can occur only as follows: $
        11 \cdot 19 \mid \Delta(n-4,5)$ for $ (k,a)=(5, 358),$ $
        11 \cdot 13 \mid \Delta(n-4,5)$  for $(k,a)=(5,842).$ Thus, from the above $31$ possible values of $n$ for $k=5$, we obtain the corresponding triples $(n,a,k)$ which we denote by the set $\Tilde{T}$, where $
    \Tilde{T}:=\{(13, 842,5), (14, 842,5), (15, 842,5), (22, 358,5), (23, 358,5), (26, 842,5), (58, 358,5) ,
    (66, 842,5),\\(156, 842,5) ,(156, 358,5) \}.$
    Therefore, we can assume,  $\omega_{11}\bigl(\Delta(n-k+1,k)\bigr) \leq 1$. Furthermore, by Lemma \ref{lemmak=1to5}, if $11\leq p\leq 31$, then, $\nu_p(\Delta(n-k+1,k))\leq 1$. Starting from the $164$ pairs $(n,k)$ obtained above, we now reduce this list using these conditions. This leaves $31,14,$ and $6$ possible values of $n$ for $k=3,4,5$, respectively, giving a total of $51$ remaining pairs $(n,k).$  Thus, the set of pairs $(n,k)$ requiring further consideration is substantially reduced. For each of these remaining pairs $(n,k)$, we next consider the triples $(n,a,k)$ with $50<a\le 1000$. We check that
    $P\bigl(\Delta(n-k+1,k)\bigr) \geq 11$
    for all $n \geq 11$, except for $(n,k)=(16,3)$.
    By Lemma \ref{lemmak=1to5}, for a prime divisor $p \geq 11$ of  $\Delta(n+a-k+1,k)$,
    $p^2 \parallel (a+k-1)(a+k)$ for $k=3,4$ and $p^2 \parallel (a+k-2)(a+k-1)(a+k)$ for $k=5$. Consequently, for each fixed $k$ with $3 \leq k \leq 5$ and each $n \geq 11$, we collect all integers $a$ satisfying $50 < a \leq 1000$.  These conditions reduce the problem to a finite set of $87$ pairs $(n,a)$ for $k=5$, $96$ pairs $(n,a)$ for $k=4$, and $186$ pairs $(n,a)$ for $k=3$. Together with the $10$ triplets $(n,a,k)$ contained in $\Tilde{T}$, this yields a total of $379$ triplets $(n,a,k)$ under consideration. For all such triples $(n,a,k)$, with the exception of $(65,167,3)$ and $(25,527,3)$, we observe that
    $P\bigl(\Delta(n+a-k+1,k)\bigr)\ge 37$.
   Therefore, we can apply Lemma \ref{n+a- 1<=k<=5}$(i)$ and thus we are left with the cases $(n,a,k)\in\{ (25, 527,3), (56,845,3), (65, 167,3) \},$ when $n\geq 11$ and $(n,k)\neq (16,3)$. We construct a Newton polygon with a suitable prime for the cases with $n\geq 25$ as given in the following table and conclude by Lemma \ref{fil1} that $\mathcal{L}_n^{(\alpha)}(x)$ doesn't have a factor of degree $k=3.$ 
        \begin{center}
            \scalebox{0.7}{
            \begin{tabular}{|c|c|c|c|}
            \hline
                $(n,a,k)$ & $p$ & Vertices of $NP_p(\mathcal{L}_n^{(a)}(x))$& Slope of the right-most edge  \\
                \hline
                 (25, 527, 3)& 29& (0, 0), (1, 0), (25, 1)&$ {1}/{24}$\\
                 \hline
                 (56, 845, 3)& 17 &$(0, 0), (34, 2), (51, 4), (56, 5)$&$ {1}/{5}$\\
                 \hline
                 (65, 167, 3)& 29& (0, 0), (58, 2), (65, 3)&${1}/{7}$\\
                 \hline
            \end{tabular}}
        \end{center}
    Now consider $ n \leq 10$ and $ n = 16$ when $ k = 3$. We again use Lemma  \ref{n+a- 1<=k<=5} as before. Since $P(\Delta(n-2k+1,k))< 37$ for these $n$, we should have $P(\Delta(n+a-k+1,k))<37$ as well. That is, there are $228$ pairs of admissible values of $(n,a)$  for $k=3$, $36$ pairs of $(n,a)$ for $k=4$ and $2$ pairs of $(n,a)$ for $k=5$. After excluding the values of $k$ from the triplets, this leaves a total of $97$ distinct pairs $(n,a)$. Since the corresponding values of $n$ are small, and the factorisation of $\mathcal{L}_n^{(a)}(x)$ is readily computable for such $n$, we explicitly computed the factorisations for these $97$ pairs and verified that, in each case, the polynomial is irreducible.
    \end{proof}

    \begin{lemma}\label{nofactor2}
        Let $n\geq 4$ and $50<a\leq 1000$. Then $\mathcal{L}_n^{(a)}(x)$ has no factor of degree $2$.
    \end{lemma}
    \begin{proof}
    Suppose that $\mathcal{L}_n^{(a)}(x)$ has a
    factor of degree $2$.   From  \cite[Table IA]{lehmer} we have $P(n(n-1))> 31$ for $n > n_2$, where $n_2=1611308700$. 
    Hence, by Lemma \ref{pmax1+2k}, it follows that $n < n_2$. Let $n > 10^5$. Now we consider $n$ such that $P(n(n-1))\leq 31$. For such $n$, we refer from \cite[Table IB]{lehmer} that $\omega_{11}(n(n-1))\geq 2$. Hence by Lemma \ref{lemmak=1to5},  we can assume that $n\leq 10^5, P\bigl(n(n-1)\bigr) \leq 31$ and $   \omega_{7}\bigl(n(n-1)\bigr) \leq 1.$ Using \cite[Table IA]{lehmer}, we compile a set consisting of $59$
    such integers $n$. 

    Let $11\leq P(n(n-1))\leq 31$.
    By Lemma \ref{lemmak=1to5}, if $p \geq 11$ is a prime divisor of
    $n(n-1)$, then $p^2 \mid (a+2)$ and $
    \nu_p\bigl(n(n-1)\bigr)=1$. There are $32$ such possible values of $n$. For each such $n$, we determine all integers $a$ such that for every prime $p\geq 11$ satisfying  $p\mid n(n-1)$, we have $p^2\mid (a+2)$. This gives 118 pairs $(n,a)$. We then filter these pairs using Lemma \ref{n+a- 1<=k<=5}$(i)$ and $(iii)$. It can be checked that, for all these pairs, $P((n+a)(n+a-1))\geq 11$. We first consider the pairs $(n,a)$ for which $P((n+a)(n+a-1))\geq 37$. There are $106$ such pairs. By Lemma \ref{n+a- 1<=k<=5}$(i)$, we have $P((n+a)(n+a-1))\leq P((n-3)(n-2))$. Only $8$ pairs of $(n,a)$ satisfy this condition. Furthermore, if $p\geq 37$ and $p\mid(n+a)(n+a-1)$, then Lemma \ref{n+a- 1<=k<=5}$(i)$ gives $p\mid(a+1)(a+2)$. However, none of the $8$ pairs satisfies this additional condition. Hence, it remains to consider the 12 pairs of $(n,a)$ for which $P((n+a)(n+a-1))\leq 31$. We now consider the prime divisors $p$ of $(n+a)(n+a-1)$ with $11\leq p\leq 31$. If $11\leq p\leq 31$ and $p\mid(n+a)(n+a-1)$, then Lemma \ref{n+a- 1<=k<=5}$(iii)$ gives $p\mid(a+1)(a+2)$ or $p^2\mid(a+3)(a+4)$. But none of these pairs satisfies this condition.  Let $P(n(n-1))=7$. Using \cite[Table IA]{lehmer}, there are exactly $13$ such integers $n$. By Lemma \ref{lemmak=1to5}, any corresponding $a$ must satisfy either $7^2 \mid (a+2)$ or $7^3 \parallel (a+4)$. Among the integers $a$ with $50 < a \leq 1000$, precisely $21$ values satisfy this condition. Consequently, we are left with $273$ admissible pairs $(n,a)$. We apply Lemma \ref{n+a- 1<=k<=5}$(i), (iii)$ similar to the case $P(n(n-1))\geq 11$ and obtain the pair  $(n, a) = (36,339)$.

     Let $P(n(n-1))\leq 5$. Then, $n\in\{81, 25, 16, 10, 9, 6, 5, 4\}$. We first consider $n=81$. Repeating the above process using Lemma \ref{n+a- 1<=k<=5}, we obtain $(n,a)\in\{(81, 115),(81, 235), (81, 544), (81, 648)\}$. Thus, there are $4$ admissible pairs $(n, a)$ corresponding to $n=81$. Since \(P(\Delta(n-3,2))<37\) for all \(n\neq 81\), Lemma \ref{n+a- 1<=k<=5}$(i)$ implies that $
    P\bigl((n+a)(n+a-1)\bigr)<37$ for all \(n\neq 81\). Repeating the above process, we are left with $51$ pairs of $(n,a)$ satisfying $P(n(n-1))\leq 5$ and $n\neq 81$. For these $51$ small cases, since $n$ is small, we directly compute and verify that $\mathcal{L}_n^{(a)}(x)$ has no factor of degree $2$.
The only remaining cases are $\{(36, 339), (81, 115),(81, 235),  (81, 544), (81, 648)\}$. For each of these cases, we construct the Newton polygon with respect to a suitable prime and conclude by Lemma \ref{fil1} that $\mathcal{L}_n^{(\alpha)}(x)$ has no factor of degree $2$. 
     \begin{center}
        \scalebox{0.7}{\begin{tabular}{|c|c|c|c|}
        \hline
        $(n,a)$& $p$ & Vertices of $NP_p\left(\mathcal{L}_n^{(a)}(x)\right)$& Slope of the right-most edge  \\
        \hline
        $(36, 339)$ & $3$ & $(0, 0), (9, 4), (36, 17)$ & ${13}/{27}$\\
        \hline
       $(81, 115)$ &$7$&$(0, 0), (49, 8), (77, 13), (81, 14)$ &$1/4$\\
        \hline
        $(81, 235)$ &$7$&$(0, 0), (1, 0), (8, 1), (15, 2), (22, 3), (71, 11), (78, 13), (81, 14)$ &${1}/{3}$\\
        \hline
        $(81, 544)$ &$3$&$(0, 0), (81, 40)$ &${40}/{81}$\\
        \hline
        $(81, 648)$ &$7$&$(0, 0), (1, 0), (8, 1), (15, 2), (22, 3), (29, 4), (78, 13), (81, 14)$ &${1}/{3}$\\
        \hline
        \end{tabular}}
        \end{center}    
    \end{proof}
\section{\texorpdfstring{Proof of Theorem \ref{main} with $k=1$}{Proof of Theorem \ref{main} with k=1} }
    \noindent
    Assume that $\mathcal{L}_n^{(a)}(x)$ has a linear factor. Recall that $r_1$ is the length of the right-most edge in the $NP_p\left(\mathcal{L}_n^{(a)}(x)\right)$. In this section, we derive conditions on the primes $p$ dividing $n$.
     For every $p\mid n(n+a)$ by \eqref{phij>j/k} we have, 
    \begin{align}\label{phij>j}
    \nu_p(\Delta(a+1,r_1))\geq \nu_p(\Delta(a+1,r_1))-\nu_p\left(\binom{n}{r_1}\right)\geq r_1.        \end{align}

    \begin{lemma}\label{lemmak=1}
    Let $n\geq 2$ and $50<a\le 1000$. Suppose that $\mathcal{L}_n^{(a)}(x)$ has a factor of degree $k=1$. Let $p$ be a prime with $p\mid n$. Then $p\leq 31$ and the following holds:
    \begin{enumerate}[$(i)$]
    \item If $11\leq p\leq 31$, then 
    $p^2\parallel (a+1)\ \text{and} \ \nu_p(n)=1.$
    \item 
    If $p=7$ or $p=5$, then 
    $p^{j+1}\mid(a+j),$ and $\begin{cases}
    \nu_p(n)\leq \nu_p(a+1)-1,&j=1,\\
    \nu_p(n)=\nu_p(a+j)-j,&2\leq j\leq j_p,
    \end{cases}$
    where $j_7=3$ and $j_5=4$.
    
    \item If $p=3$, then one of the following holds:
    \begin{align*}
    & 3\parallel (a+1),\ 
    \nu_3(n)\le \nu_3(a+4)-3\ \text{and}\ 
    \nu_3(n)\le \max\{\nu_3(a+4),\nu_3(a+7)\}-5,\quad\text{or}\\
    &3^2\mid (a+1)\quad \text{and}\quad \nu_3(n)\le \nu_3(a+1)-1;\\
    & 3\parallel (a+2)\  \text{and}\ \nu_3(n)\le \nu_3(a+5)-4, \quad\text{or}\quad 3^3\mid (a+2)\quad \text{and}\quad \nu_3(n)= \nu_3(a+2)-2;\\ &3\parallel (a+3)\ \text{and}\ \nu_3(n)\le \nu_3(a+6)-4\quad\text{or}\quad 3^3\mid (a+3)\quad \text{and}\quad \nu_3(n)= \nu_3(a+3)-2,\quad\text{or}\\
    &3^6\mid (a+3)\quad \text{and}\quad \nu_3(n)= \nu_3(a+3)-5.
    \end{align*}
    \end{enumerate}
    \end{lemma}

    \begin{proof} By Lemma \ref{pmax1+2k}, we have $p\leq 31$. We now consider the possible values of $p.$

    \noindent
    $(i).$ Suppose $11 \leq p \leq 31$. Then the assertion  follows directly from Lemma \ref{lemmak=1to5}.\\
    \noindent
    $(ii).$   Let $p\in\{5, 7\}$. We consider the conditions of Lemma \ref{slope1}. Since $k=1$ and $p\mid n$, we have $l=0$ and $c=p-1.$ Further, since $a \leq 1000,$ we have $c_0\leq p-1$ and hence $\frac{c-1}{p-1}(p^{v_0}+p)\geq \frac{2p(p-2)}{p-1}>c_0$. Let $\alpha=\nu_p(t_0+1)$. Then by Remark \ref{remark}$(i)$, $\alpha\leq c_0-v_0$. Therefore, $\alpha\leq c_0-v_0\leq p-2\leq \frac{c-1}{p-1}p^{v_0}$. 
    Thus by Lemma \ref{v0gequ0+1}, we have $r_1=j_0$, $v_0\geq u_0+1$ and $j_0\geq l+1$. Hence $\nu_p(\Delta(a+1,p))=\nu_p(a+j_0)=v_0\geq 2$. Since $a\leq 1000$, we have
    \begin{align}\label{equation-v0u0}
    \begin{cases}
    & 2\leq v_0\leq 3 \quad\text{ and }\quad 1\leq v_0-u_0\leq 2 ,\quad\text{ when }\quad p=7, \\
    & 2\leq v_0\leq 4\quad \text{ and }\quad 1\leq v_0-u_0\leq 3,\quad\text{ when }\quad p=5 .
    \end{cases}
    \end{align}
    By Lemma \ref{v0gequ0+1}, we have $r_1=j_0$. Then from \eqref{phij>j} and Remark \ref{remark}$(iii)$, we have $j_0\leq v_0-u_0.$
    Assume that $r_2=p$. Then by Lemma \ref{slope1}, we have $
    l+1\leq j_0<p\leq j_0+1.$ Thus, $p=j_0+1$. As $c_0\leq p-1$, we have $v_0\leq p-1$ and hence $j_0\leq p-2$, which is a contradiction.
    Therefore $r_2 \neq p$ for $p\in\{5,7\}$.  Let $p=7$. Since $r_2 \neq p$,  Lemma \ref{v0gequ0+1} applies.
    If $(v_0-u_0,j_0)=1$, then, writing $\frac{v_0-u_0}{j_0}=\frac{r}{s},$ we have \(s=j_0\). Since $k=1$ and $l=0$, we derive $j_0=1$. In this case, since $v_0\geq u_0+1$, we get from \eqref{phij>j}, $7^2\mid(a+1)\quad\text{and}\quad\nu_7(n)=u_0\leq v_0-1=\nu_7(a+1)-1.$ Let $(v_0-u_0,j_0)>1$. Thus $v_0-u_0 =2$ from \eqref{equation-v0u0} implying $j_0=2$ since $j_0\leq v_0-u_0$. By \eqref{equation-v0u0}, this gives $v_0=3, u_0=1.$ Hence $7^3\parallel(a+2) \ \text{and} \ \nu_7(n)=\nu_7(a+2)-2=1.$ This proves the assertion for \(p=7\). The case $p=5$ is similar. Since $r_2\neq p$, Lemma \ref{v0gequ0+1} applies. Using $j_0\leq v_0-u_0$ and \eqref{equation-v0u0}, we obtain the possible cases $j_0=1,\quad (v_0-u_0,j_0)=(2,2),\quad \text{or}\quad (v_0-u_0,j_0)=(3,3).$ These give, respectively, $5^2\mid(a+1), 5^3\mid(a+2),5^4\parallel(a+3),$ together with the corresponding assertions on \(\nu_5(n)\). \    
    
    \noindent
    $(iii).$ Let $p=3$. Since $a \leq 1000$ and $j_0 \leq 3$ we have $\nu_3(\Delta(a+1,3))=\nu_3(a+j_0)=v_0\leq 6.$
    We first assume that $v_0\geq 2$. Since $k=1,$ we have $c=2.$ Also, $c_0 \leq 6$ and $p^{v_0}+p\geq p^2+p=12$.
    Hence $\frac{c-1}{p-1}(p^{v_0}+p)\geq \frac{12}{p-1}=6\geq c_0.$ Let $\alpha=\nu_p(t_0+1)$. If $p^{v_0}(c-1)<\alpha(p-1)$, then by Remark \ref{remark}$(i)$,  $c_0-v_0\geq\alpha>p^{v_0}\frac{c-1}{p-1}=\frac{p^{v_0}}{p-1}\geq \frac{9}{2},$ which is impossible since $v_0\geq 2$ and $c_0-v_0\leq 4$. Hence $\alpha(p-1)\leq p^{v_0}(c-1)$ and by Lemma \ref{v0gequ0+1}, $v_0\geq u_0+1$, $j_0\geq l+1$ and $r_1=j_0$ except in the exceptional case when $v_0-u_0=2$ and $j_0=2$. Further, $1\leq v_0-u_0\leq 5.$ In the exceptional case when
    $v_0-u_0=2$ and $j_0=2$, we have $
    3^3 \mid (a+2)$ and $ 
    \nu_3(n)=\nu_3(a+2)-2.$ Hence, we may assume from now on that $r_1=j_0.$ We consider the following cases depending on $r_2$.
    
    \noindent
    \textbf{Case I: $r_2\neq p$}. Let $j_0=p$. Since \(p=3\), this means \(j_0=3\). If $(v_0-u_0+1, p)=1$, then $j_0\leq k$ by Lemma \ref{v0gequ0+1}, which is a contradiction. Hence $v_0-u_0+1\in\{3, 6\}$. This gives the following alternatives: $
    3^3 \mid (a+3) \ \text{and} \
    \nu_3(n)=\nu_3(a+3)-2 \ \text{or} \ 3^6\mid (a+3) \ \text{and} \
    \nu_3(n)=\nu_3(a+3)-5.$ Let $j_0=2$. 
    If $v_0-u_0=2$, then $r_1=p$.
    So, let $v_0-u_0\neq 2$. 
    Then $\phi^{(1)}=\frac{v_0-u_0}{j_0}$ by Remark \ref{remark}$(iii)$. 
    Assume that $(v_0-u_0, j_0)=1$. Then $j_0\leq k$, which is a contradiction. Therefore $v_0-u_0=4$ and $3^5\mid (a+2)\ \text{and}\  \nu_3(n)=\nu_3(a+2)-4.$ Let $j_0=1$. Then $
    3^2 \mid (a+1)$ and $
    \nu_3(n)\le \nu_3(a+1)-1.$
    
    \noindent
    \textbf{Case II: \(r_2=p\)}. Then by Lemma \ref{slope1}$(ii)(c)$, \(j_0<p\le j_0+1\) and hence, we must have \(j_0=2\). If $v_0-u_0=1$, then we have $\phi^{(1)}=\frac{1}{2}$ and $\phi_p(p)=\frac{2}{3}$, which is a contradiction to $\phi^{(1)}>{\phi_p(p)}$. Moreover,
    \(v_0-u_0\neq 2\), as that would imply $r_1=p$, and hence \(v_0-u_0\ge 3\). Therefore $
    3^4 \mid (a+2)$ and $
    \nu_3(n)\le \nu_3(a+2)-3.$
    
    Thus, it remains to consider the case $v_0=1.$ Then $3\parallel (a+j_0)$ for some $1\leq j_0\leq 3$. Consider $\phi_j(p)$ as defined in \eqref{phij}.  Since \(3\parallel(a+j_0)\), we have $\phi_j(3)<1\quad\text{for all}\quad j<j_0+3.$ As \(\phi_{r_1}(3)\geq 1\) by \eqref{phij>j}, it follows that $r_1\geq j_0+3.$ We proceed similarly as in the proof of  Lemma \ref{n+a- 1<=k<=5}. Let $r_1=j_0+i_1$, where $i_1\geq 3$. Define $v=\max\{\nu_3(a+j_0+i):0\leq i\leq i_1 \}.$
    Then, $r_1\geq 3^{v}-a$ and
       $ \nu_3(\Delta(a+1, r_1))=\nu_3(\Delta(a+j_0, i_1+1))\leq v+\nu_3(i_1!).$
        By Lemma \ref{sigma}, we have $\phi_{r_1}(3)=\frac{\nu_3(\Delta(a+1,r_1))-\nu_3\left(\binom {n}{r_1}\right)}{r_1}< \frac{v}{r_1}+\frac{\nu_3(i_1!)}{i_1}<\frac{v}{r_1}+\frac{1}{2}.$ 
        Since $\phi_{r_1}(3)\geq 1$, it follows that $r_1< 2v$. Also, since $r_1\geq 3^v-a$, we have $3^v-a< 2v$. As $50< a\leq 1000$, this inequality is possible only when $v\leq 6$. Therefore, $r_1< 12$ and we obtain $j_0+3\leq r_1< 12$. We have $$
        \nu_3(\Delta(a+1,r_1))
        \leq 2+\max\{\nu_3(a+j_0+3),\nu_3(a+j_0+6)\}\leq 8,\ \text{ for} \ \max\{j_0+6, 8\}\leq r_1< j_0+9,\ \text{and}$$
        $$\nu_3(\Delta(a+1,r_1))\leq 3+\max\{\nu_3(a+j_0+3),\nu_3(a+j_0+6),\nu_3(a+j_0+9)\}\leq 9\ \text{for}\ j_0+9\leq r_1< 12.$$ Further, by Lemma \ref{lucas}, $\nu_3\left(\binom{n}{r_1}\right)\geq 1\quad\text{if}\quad r_1=8$. Hence \eqref{phij>j} cannot hold in this range and we can assume that $j_0+3\leq r_1<\max\{j_0+6,8\}$ and $1 \leq j_0 \leq 3$.     Suppose first that $j_0=1$.  Then $4\leq r_1\leq 7$. Further,
        $\nu_3(\Delta(a+1, r_1))=1+\nu_3(a+4)$ for $4\leq r_1\leq 6$ and $\nu_3(\Delta(a+1, r_1))\leq 2+\max\{\nu_3(a+4),\nu_3(a+7)\}$ for $r_1=7$. 
        Moreover, we have $\nu_3\left(\binom{n}{r_1}\right)\geq\nu_3(n)+i,$where $i=0$ when $r_1\in\{4,5, 7\}$ and $i=1$ when $r_1=6$. Therefore \eqref{phij>j} can hold only if  $\nu_3(a+4)-\nu_3(n)\geq 3\ \text { or } \ \max\{\nu_3(a+4), \nu_3(a+7)\}-\nu_3(n)\geq 5.$
        Suppose next that \(j_0=2\). Then \(5 \le r_1 < 8\). Further, $\nu_3(\Delta(a+1, r_1))=1+\nu_3(a+5)$ and $\nu_3\left(\binom{n}{r_1}\right)\geq \nu_3(n)+i$, where $i=0$ when $r_1=5$ and $i=1$ when $r_1\in\{6,7\}$.
        \noindent Therefore \eqref{phij>j} can hold only if  $\nu_3(a+5)-\nu_3(n)\geq 4.$ Suppose $j_0=3$. 
        Then $6\leq r_1<9$. Consider the coefficients of $x^i$ in $\mathcal{L}_n^{(a)}(x)$. We have
        $|a_7|=|a_6|\frac{(n-6)}{7(a+7)}$ and $|a_8|=|a_7|\frac{(n-7)}{8(a+8)}.$ Since $3\mid(n-6)$ and $3\nmid 7\cdot 8(a+7)(a+8)$, we get $\nu_3(a_6)<\nu_3(a_7)=\nu_3(a_8)$. Therefore, since $6\leq r_1<9$, we must have $r_1=6$. Now $\nu_3(\Delta(a+1, 6))=\nu_3(a+3)+\nu_3(a+6)=1+\nu_3(a+6),$ and $\nu_3\left(\binom{n}{6}\right)=\nu_3(n(n-3))-2\geq \nu_3(n)-1.$ 
        Thus, \eqref{phij>j} can hold only if  $\nu_3(a+6)-\nu_3(n)\geq 4.$
    \end{proof}

    \begin{lemma}\label{generalboundp}
        Let $p$ be a prime and $m\geq 1$ be such that $p^m\leq a<p^{m+1}$. If $\nu_p(n)\geq m+1$, then $\phi_j(p)<\frac{1}{p-1}$ for all $1\leq j\leq n$.
    \end{lemma}
    \begin{proof}

        For $1\leq j\leq n$, we have $\nu_p\left(\binom nj\right)=\nu_p(\Delta(n-j+1,j))-\nu_p(j!)\geq \nu_p(n)+\nu_p((j-1)!)-\nu_p(j!)=\nu_p(n)-\nu_p(j)$. Thus, from \eqref{phij} and Lemma \ref{sigma}, we have         
        \begin{align*}\label{1/p-1}
        j\phi_j(p)
        \leq\frac{j+\sigma_p(a)-\sigma_p(a+j)}{p-1}-(\nu_p(n)-\nu_p(j))\numberthis.
        \end{align*}
     Let $\nu_p(j)>m$. Then since $p^m\leq a<p^{m+1}$, $\sigma_p(a+j)=\sigma_p(a)+\sigma_p(j)$. Hence, from \eqref{1/p-1},
        $j\phi_j(p)\leq\frac{j-\sigma_p(j)}{p-1}<\frac{j}{p-1}.$ Now we consider the case when $\nu_p(j)\leq m$. Suppose $\nu_p(j)=w$ and let $j=\sum\limits_{l\geq w}^{}t_lp^l$ be the base-$p$ expansion of $j$, where, $0\leq t_l\leq p-1$. Let $a=\sum\limits_{i=0}^{w-1}r_ip^i+\sum\limits_{l=w}^ms_lp^l$
        be the base-$p$ expansion of $a$.
        Then $a+j=\sum\limits_{i=0}^{w-1}r_ip^i+\sum\limits_{l\geq w}(s_l+t_l)p^l,$ where $s_l=0$ for $l>m$, and $\sigma_p(a+j)\geq\sum\limits_{i=0}^{w-1}r_i+1$ implying 
        $\sigma_p(a)-\sigma_p(a+j)\leq\left(\sum\limits_{l=w}^m s_l\right)-1\leq (m-w+1)(p-1)-1.$ Hence, from \eqref{1/p-1} and using $\nu_p(n)\geq m+1$, we get $
            j\phi_j(p)
            \leq\frac{j+(m-w+1)(p-1)-1}{p-1}-(\nu_p(n)-\nu_p(j))
            \leq \frac{j-1}{p-1}<\frac{j}{p-1}.$   
    \end{proof}     
    In the linear case, when $2\mid n$, or  when $3\mid n$ with $3\parallel\Delta(a+1, 3)$, Lemma \ref{slope1} cannot be applied. Therefore, we treat these cases separately. For $\nu_2(n)$, we obtain a better bound directly from the Lemma \ref{generalboundp}, whereas for $\nu_3(n)$, we are able to improve the upper bound given by Lemma \ref{generalboundp}.
    
    \begin{lemma}\label{w-k1}
        Let $n\geq 2$ and $50<a\leq 1000$. If $\mathcal{L}_n^{(a)}(x)$ has a linear factor, then
        \begin{enumerate}[$(i)$]
            \item $\omega_5(n)\leq 1$,
            \item $n$ is of the form $2^\alpha\cdot3^\beta\cdot p^{\delta_p}$, where $\alpha\leq\lfloor\log_2(a)\rfloor\leq 9$, $\beta\leq 5$, $5\leq p\leq 31$ and $\delta_5\leq 3$, $\delta_7\leq 2$,  $\delta_{p}\leq 1$ for $11\leq p\leq 31$,
            \item either $P(n)<11$ or $P(n+a)<11$.
        \end{enumerate}
    \end{lemma}
    \begin{proof}
    \noindent
    $(i).$ Let $\omega_5(n)\geq 2$. 
        We first assume that $5\cdot 7\mid n$. By Lemma \ref{lemmak=1}$(ii)$, for the prime $7,$ either $7^2\mid(a+1)$ or $7^3\mid(a+2)$. If $7^2\mid(a+1)$, as $7^2\cdot5^2>1000$, 
        $5^3\mid(a+2)(a+3)$. This contradicts Lemma \ref{equationmm'}. Therefore, we must have $7^3\mid(a+2)$. That is, $a\in\{341, 684\}$. For both of these values of $a$, we have $5^2\nmid(a+1)(a+2)(a+3)$ contradicting Lemma \ref{lemmak=1}$(ii)$. Next, assume that $5\mid n$ and $p \mid n$ for some $p \geq 11$. By Lemma \ref{lemmak=1}$(i)$, we have $p^2 \mid (a+1).$ The alternative $5^2 \mid (a+1)$ is impossible, since $5^2 \cdot p^2 >1000$. Therefore, \ref{lemmak=1}$(ii)$ for $p=5$ gives $5^3 \mid (a+2)(a+3),$ which contradicts Lemma \ref{equationmm'}. It remains to consider the case in which $5 \nmid n$. Since $\omega_5(n) \geq 2,$ we have $\omega_7(n) \geq 2.$  By Lemma \ref{lemmak=1to5}, we also have $\omega_{11}(n)\leq 1$. Hence $7\mid n$ and $p \mid n$ for some $p \geq 11.$ By Lemma \ref{lemmak=1}$(i)$, this gives $p^2 \mid (a+1).$ Since $7^2 \mid (a+1)$ is impossible, Lemma \ref{lemmak=1}$(ii)$ for $p=7$ gives $7^3 \parallel(a+2)$. This contradicts  Lemma \ref{equationmm'}. Therefore $\omega_5(n)\leq 1$.
        
    \noindent
    $(ii).$ This follows from Lemmas \ref{lemmak=1}, \ref{generalboundp} and part $(i)$ of Lemma \ref{w-k1} proved above.
    
    \noindent
    $(iii).$ Let $P(n) \geq 11$ and $P(n+a) \geq 11$. Then there exist primes $p,q \geq 11$ such that $p \mid n, q \mid (n+a)$. By Lemma \ref{lemmak=1} $(i)$, we have $p^2\mid(a+1)$.  If $q>31$, then Lemma \ref{n+a- 1<=k<=5} gives $q\mid (a+1)$. Hence $p^2q \mid (a+1)$ which is impossible since $p\geq11, q>31$ and $a \leq 1000$. Therefore, we may assume that $p,q\leq 31$. By Lemma \ref{n+a- 1<=k<=5}, either $q\mid(a+1)$ or $q^2\mid(a+2)$. If $q \mid (a+1),$ then $ p^2 q \mid (a+1)$, which is impossible since $p^2q \geq 11^3>1000$. Thus $q^2\mid(a+2)$, which  contradicts Lemma \ref{equationmm'}. Therefore either $P(n)<11$ or $P(n+a)<11$.       
    \end{proof}
    Among the pairs $(n,a)$ where $\mathcal {L}_n^{(a)}(x)$ may have a linear factor, there are pairs with large $n$ values for which the slope of the right-most edge is greater than or equal to 1 for any prime that meets the conditions outlined in Lemma \ref{slope1}. This means we cannot conclude anything about the linear factor, and direct polynomial computation is also infeasible. The following discussion draws inspiration from a recent work of Banerjee \cite{banerjee} that uses the theory of roots given by Szeg\`{o} \cite{szego}. Nevertheless, the arguments presented here are carried out independently and are adapted and strengthened to suit the specific needs of our problem. The following lemmas give a lower bound and an upper bound for the roots of $\mathcal L_n^{(a)}(x)$.

    \begin{lemma}\label{szego}
   Let $n,a\in\mathbb Z^+$. All the roots of  $\mathcal {L}_n^{(a)}(x)$ are positive and the largest zero of the polynomial $\mathcal {L}_n^{(a)}(x)$ satisfies $x_n\leq 2n+a+1+\left\{(2n+a+1)^2+\frac{1}{4}-a^2\right\}^\frac{1}{2}.$
    \end{lemma}
    \begin{proof}
        By \cite[Theorem 6.31.3]{szego}, the least zero of $\mathcal {L}_n^{(a)}(x)$, say $x_1$, is bounded below by  $\dfrac{(j_1/2)^2}{n+(a+1)/2}$, where $j_1$ is the least positive zero of Bessel's function $J_a(x)$. (see  \cite[Equation 1.71.1]{szego}). Since $a>0$ we have $x_1>0$ and hence all the roots of $\mathcal {L}_n^{(a)}(x)$ are positive. The upper bound on $x_n$ follows from  \cite[Theorem 6.31.2]{szego}.
    \end{proof}
\subsection{Proof of Proposition \ref{root}:}
       Let $\xi$ be an integer root of $\mathcal{L}_n^{(a)}(x)$. Then
    \begin{equation}\label{rooteq}
        \xi^n - n(n+a)\xi^{n-1}+\frac{n(n-1)}{2}(n+a)(n+a-1)\xi^{n-2}+ \sum_{j=0}^{n-3}(-1)^ja_j \xi^j =0.
    \end{equation} 
By Lemma \ref{ndivaj}$(i)$, $n\mid a_j$ for $0\leq j\leq n-1$. Thus by \eqref{rooteq}, $n\mid \xi^n$, $(n+a)\mid\xi^n$ and hence $\operatorname{Rad}(n(n+a))\mid \xi$. Similarly, by Lemma \ref{ndivaj}$(i)$, we obtain $(n+a-1)\mid a_j$ and $(n-1)\mid a_j$ for $0\leq j\leq n-2$. 
    Let 
    \begin{align}\label{d1}
    d_1 = \gcd (\xi^{n-1}, n-1),\ n_1= \frac{n-1}{d_1},\quad       d_2 = \gcd (\xi ^{n-1}, n+a-1),\  n_2= \frac{n+a-1}{d_2}.
    \end{align} Let $\gcd(n_1, n_2)=n_0$.  Then $n_0\mid a$. Let $\gcd(n_1,a)=n_0b_1, \gcd(n_2,a)=n_0b_2 \ \text{and} \ n_1=n_0b_1n_3, n_2=n_0b_2n_4.$  Note that $\gcd(b_1n_3, b_2n_4)=1$. Since $n(n+a)=(n+a-1-a+1)(n+a-1+1)$, we have
    \begin{align*}
    \xi^{n-1}(\xi-(a+1))&\equiv\xi^{n-1}(\xi-n(n+a))\pmod{n-1}\equiv 0\pmod {n-1}, \ \text{and} \\
         \xi^{n-1}(\xi-(-a+1))&\equiv\xi^{n-1}(\xi-n(n+a))\pmod {n+a-1}\equiv0\pmod {n+a-1}.
    \end{align*}
Hence 
    \begin{align}\label{xi}
    \xi\equiv a+1\pmod {n_1}\quad \text{and}\quad
    \xi\equiv -a+1 \pmod {n_2}.\end{align}
For $k_1, k_2\in\mathbb Z$, let \begin{align}\label{xi1}
        \xi=a+1+n_1k_1\quad \text{and}\quad\xi=-a+1+n_2k_2.
    \end{align}
Using $\xi=-a+1+n_2k_2$ in  \eqref{xi}, we have $n_2d_2k_2\equiv2ad_2\pmod{n_1}$. Since    $d_2n_2=n+a-1\equiv a \pmod {n_1}$,
    we have $ak_2\equiv 2ad_2\pmod{n_1}$
implying 
    $ k_2\equiv 2d_2\pmod{n_3}.$ Let $k_2=2d_2+n_3k_0$ for some $k_0\in\mathbb Z$. Substituting in \eqref{xi1}, we obtain
    \begin{align}\label{k0}
    \xi&=-a+1+n_2(2d_2+n_3k_0)=-a+1+2(n+a-1)+n_2n_3k_0=2n+a-1+n_2n_3k_0.
    \end{align} Similarly, by using
    $a+1+n_1k_1\equiv -a+1\pmod {n_2}$, we get\begin{align}\label{k0tilde}
    \xi=a+1+n_1(2d_1+n_4\tilde{k_0})=a+1+2(n-1)+n_1n_4\tilde{k_0}=2n+a-1+n_1n_4\tilde{k_0},
    \end{align}
    for some $\tilde{k_0}\in\mathbb Z$. Comparing \eqref{k0} and \eqref{k0tilde}, we get    $$n_2n_3k_0=n_1n_4\tilde{k_0}\implies n_2\frac{n_1}{n_0b_1}k_0=n_1\frac{n_2}{n_0b_2}\tilde{k_0}\implies \frac{{k_0}b_2}{b_1}=\tilde{k_0}.$$
    Since $\gcd(b_1, b_2)=1$, $b_1\mid k_0$. Therefore
    $
 \xi=2n+a-1+n_2n_3k_0=2n+a-1+n_2n_1\frac{k_0}{n_0b_1}
        =2n+a-1+\frac{(n-1)(n+a-1)}{d_1d_2}\gamma,$
    where $\gamma=\frac{k_0}{n_0b_1}\in\mathbb Q$. Further, if $\gcd(n_1, n_2)=1$, then $\gamma\in \mathbb Z$. In that case, $|\gamma|\geq 1$ and $|\xi|\geq \frac{(n-1)(n+a-1)}{d_1d_2}-(2n+a-1)\geq (n+a-1)\left(\frac{n-1}{d_1d_2}-2\right).$ By Lemma \ref{szego}, we have
    $4(n+a-1)>2n+a+1+\left((2n+a+1)^2+\frac{1}{4}-a^2\right)^\frac{1}{2}\geq|\xi|\geq (n+a-1)\left(\frac{(n-1)}{d_1d_2}-2\right),$ implying $d_1d_2>\frac{(n-1)}{6}.$
 \qed

    \begin{note}\phantomsection\label{Remark after Root}
    Note that by Proposition \ref{root}, $\operatorname{Rad}(n(n+a))\mid\xi$ and by \eqref{d1}, $\operatorname{Rad}(d_1d_2)\mid \xi$.
    \end{note}

    \begin{lemma}\label{nolinearfactorn>3}
        Let  $n\geq 2$ and  $50<a\leq 1000$. If $\mathcal{L}_n^{(a)}(x)$ have a linear factor, then $P(n)\leq 3$.
    \end{lemma}
    \begin{proof}
        Assume that $P(n)\geq 5$. By Lemma \ref{w-k1}$(ii)$, we know that \begin{align}\label{form of n}n=2^\alpha3^\beta p^{\delta_p},
        \end{align}
        where  $\alpha\leq \log_2(a)$, $\beta\leq 5$ and
        $\delta_p\leq\begin{cases}
                1,&\text{if }\ 11\leq p\leq 31,\\
                2,&\text{if }\ p=7,\\
                3,&\text{if }\ p=5.
            \end{cases}$
        
       Let $11\le p\le 31$. By Lemma \ref{lemmak=1}$(i)$, we have $p^2\parallel (a+1)$. For each prime $p$ with $11\le p\le 31$, this yields a finite set of admissible values of $a$, and for each such $a$, there exist only finitely many possible values of $n$ by \eqref{form of n}. Now we apply Lemma \ref{lemmak=1}$(iii)$ to reduce this set of values of $n$ for each $a$. Further, by Lemma \ref{w-k1}$(iii)$, we have $P(n+a)< 11$, since $P(n)\geq 11$. Consequently, there are only $17$ such pairs $(n,a)$. 
       By Lemma \ref{n+a- 1<=k<=5}\((ii)\), if \(7\mid(n+a)\), then $7^j\mid(a+j)$ for some $1\leq j\leq 3,$ and if \(5\mid(n+a)\), then $5^j\mid(a+j)$ for some $1\leq j\leq 4.$ Further, if $3\mid (n+a)$, we have $r_1< 12$. Thus, using \eqref{phij>j}, we obtain $(n,a)\in\{(22,362),(152,360),(6656,844)\}.$  Let $p\in\{5,7\}$. By  Lemma \ref{lemmak=1}$(ii)$, we have $p^2\mid (a+1)$ or $p^3\mid(a+2)$ or $p^4\mid (a+3)$. This gives a finite set of admissible values of $a$. For each such $a,$ we use \eqref{form of n}, Lemma \ref{lemmak=1}$(iii)$ and Lemma \ref{n+a- 1<=k<=5} for $k=1$ to obtain exceptional pairs of  $(n,a)$. These pairs together with the ones we obtained earlier are given by the set $\tilde{S}$, where $\Tilde{S}:=\{(5,123), (5,724),(10,374),(20,124), (22,362), (25,623), (25,999),(40,248), (40,824),(40,974),\\
        (56,538), 
        (152,360), (160,224),(200,124), (320,724),(400,624), (640,224), (896,979), (6656,844)\}.$
        
    Consider the pair  $(n,a)=(6656,844)\in \tilde{S}$. Let $\xi$ be a root of $\mathcal L_n^{(a)}(x)$. We consider \eqref{equation for root} from Proposition \ref{root}.
    Since  both $n$ and $a$ are even and since $\operatorname{Rad}(n)\mid \xi$, it follows that $\xi$ is even. Moreover, $\gcd(n-1,n+a-1)=1$. Thus, $\gamma\in\mathbb Z\setminus\{0\}$ in \eqref{equation for root}.
   The possible values of $d_1$ and $d_2$ are 
    $d_1\in\{1,5,11,55,121,605,1331,6655\}$ and $d_2\in\{1,7499\}.
    $
    \noindent Also, $d_1d_2>\frac{n-1}{6}=1109$. By Remark \ref{Remark after Root}, $\operatorname{Rad}(n(n+a))\mid\xi$ and $\operatorname{Rad}(d_1d_2)\mid \xi$. On the other hand Lemma \ref{szego} gives $\xi\le 28288$. Since $\operatorname{Rad}(n(n+a))=390$, if $d_2=7499$, then $
    7499\cdot 390>28288\ge \xi$
    which is impossible. Hence $d_2=1$ and consequently, $d_1\in\{1331, 6655\}$. A direct search over $0<\xi\leq 28288$ verifies that no integer $\xi$ of the form \eqref{equation for root} satisfies all of the stated conditions. Therefore, $\mathcal L_n^{(a)}(x)$ has no linear factor in this case. Let $(n,a)=(896,979)$. Suppose that $\xi$ is a root of $\mathcal L_n^{(a)}(x)$.  As in the above case, we consider \eqref{equation for root} from Proposition \ref{root} and we obtain that if $\gamma\neq 0$, no such $\xi$ exists.
    Hence $\gamma=0$ and  $\xi=2770$ by \eqref{equation for root}. However, $\mathcal L_n^{(a)}(2770)\neq 0$ and hence $\mathcal{L}_n^{(a)}(x)$ with $(n,a)=(896,979)$ has no linear factor.
    Let $(n,a)\in\{(200, 124), (320, 724), (400, 624), (640, 224)\}$. Suppose that $\mathcal{L}_n^{(a)}(x)$ has a linear factor and let $\xi$ be a root. Since $\operatorname{Rad}(n)\mid \xi$ and all the above values of $n$ are even, it follows that $\xi$ is even. Moreover, since all the corresponding values of $a$ are even, we have $\gamma\neq 0$. For all the above pairs of $(n,a)$ except $(400, 624)$, we have $\gcd(n+a-1, n-1)=1$. Consider $(n,a)=(400, 624)$. Then the possible values of $d_1, d_2$ are
    $d_1\in\{1, 3, 7, 19, 21, 57, 133, 399\}$ and $d_2\in\{1, 3, 11, 31, 33, 93, 341, 1023\}.$ Further, by Lemma \ref{szego}, we have $0< \xi\leq 2706$.  Among the $128$ possible values of $\xi$ in this range satisfying $\operatorname{Rad} (n(n+a)d_1d_2)\mid\xi$ for the corresponding choices of $d_1$ and $d_2$, a direct check shows that none of them satisfies \eqref{xi}.
    For $(n,a)\in\{(200, 124), (320, 724), (640, 224)\}$ we have $\gamma \in \mathbb{Z}\setminus\{0\}$. As before, we analyse all possible values of $d_1$, $d_2$, and $\xi$. A direct verification shows that no choice of $\xi$ satisfies the required conditions.
    For the remaining values of $(n,a) \in \tilde{S}$, we checked that $\mathcal{L}_n^{(a)}(x)$ are irreducible.
\end{proof}

    \begin{lemma}\label{2^alpha3^beta}
         Let $n\geq 2$ and  $50<a\leq 1000$. Then, $\mathcal{L}_n^{(a)}(x)$ has a linear factor if and only if $n=2$ and $a\in\{50< m\leq 1000: m+2\text{ is  a perfect square}\}$. 
    \end{lemma}
    \begin{proof}
        By Lemma \ref{nolinearfactorn>3}, we have $P(n)\leq 3$ and by Lemma \ref{w-k1}$(ii)$, $\nu_3(n)\leq 5$ and $\nu_2(n)\leq 9$. Let $n=2$. By Theorem \ref{n=2,3}, $\mathcal{L}_2^{(a)}(x)$ has a linear factor if and only if $a\in\{50< m\leq 1000: m+2\text{ is  a perfect square}\}$. Let $n=3$. By Theorem \ref{n=2,3}, $\mathcal L_3^{(a)}(x)$ has a linear factor and $50<a\leq 1000$, then $a\in\{51,105,122,125,247,253,340,429,497,683,726,861\}$ and we checked that $\mathcal{L}_3^{(a)}(x)$ is irreducible for these $a$.
        
        Let $n=2^\alpha3^\beta$, $0\leq \alpha\leq 9$, $0\leq \beta\leq 5$, $n>3$. By Lemma \ref{w-k1}$(ii)$, we have $\nu_2(n)\leq\log_2(a)$ and by Lemma \ref{lemmak=1}, we have a finite set of pairs $(n,a)$.  We apply Lemma \ref{n+a- 1<=k<=5} and are left $366$ with pairs of $(n,a)$.  
        For $n=512$, we have $a\in\{512, 576, 656, 664, 860, 948\}$. Let the integer root of $\mathcal{L}_{512}^{(a)}(x)$ be $\xi$. 
        Note that $\gcd(n+a-1, n-1)=1$ and $a$ is even for all our cases. Thus we have $\gamma\in\mathbb Z\setminus\{0\}$ and $d_1d_2>85$. 
        For each $a$, we consider the possibilities of $d_1$, $d_2$ and $\xi$ in the bound given by Lemma \ref{szego}. The only possibility is $(n,a,\xi)=(512, 512, 2046)$.
        We find the value of $\mathcal{L}_n^{(a)}(\xi)$ and verify that it is not a root. 
        Further, we checked that $\mathcal{L}_n^{(a)}(x)$ are irreducible for $n\leq 300$ and the pairs $(n,a)$ with $n>300$ and $n\neq 512$ are given by
        $(n,a)\in \{ (384, 480), (432, 80), (576, 160), (576, 482), (576, 674),(768, 294),
        (768, 753),(768, 480), (768, 584),\\ (1152, 723), (1536, 512)\}.$
        As in the earlier case, now we look at the roots of $\mathcal{L}_n^{(a)}(x)$. Since all $n$ are even, $\xi$, is even for all pairs $(n,a)$. If $\gamma=0$, then $(n,a)\in\{(768, 753), (1152, 723)\}$ and $\xi=2n+a-1$. But $\mathcal L_n^{(a)}(\xi)\neq 0$. Hence $\gamma\neq 0$ for all pairs. Note that for all $(n,a)$ except $(n,a)=(576, 160)$, $\gcd(n+a-1,n-1)=1$. Consider $(n,a)=(576, 160)$.
        By Lemma \ref{szego}, we have $0<\xi\leq 2616$. As in the earlier cases, we find all the possibilities of $\xi$ in this range such that $\operatorname{Rad} (n(n+a)d_1d_2)\mid\xi$ for different combinations of $d_1, d_2$, but none of them satisfy \eqref{xi}. As, $\gcd(n+a-1,n-1)=1$ for all the other pairs, we have $d_1d_2>\frac{n-1}{6}$. So, we check the possibilities of $\xi$ using Lemma \ref{szego} and Proposition \ref{root} and conclude that no such $\xi$ exists.
        Hence $\mathcal{L}_n^{(a)}(x)$ does not have a linear factor in these cases.
    \end{proof}
    
\textbf{Proof of Theorem \ref{main}}
    By Theorem \ref{n=2,3} and Lemmas \ref{nofactor6}, \ref{nofactor3-5}, \ref{nofactor2}, \ref{nolinearfactorn>3} and \ref{2^alpha3^beta}, it follows that, for $50< a\leq 1000$, $L_n^{(a)}(x)$ is irreducible except when $n=2$ and $a\in\{r^2-2:8\leq r\leq 31\}$. The remaining range is covered by \cite[Theorem 1]{laishram_shorey_integer}, which completes the proof.
\section*{Concluding Remarks:}
\vspace{-.3cm}
These techniques extend naturally to larger values of \(a\), and the bound \(a \le 1000\) is not intended to be sharp. While the methods developed here are expected to play a significant role in establishing irreducibility results over substantially wider ranges of \(a\), carrying out such extensions would require considerably more involved computations and is beyond the scope of the present paper.
\section*{Declarations of interest}
\vspace{-.3cm}
None
\section*{Acknowledgments}
\vspace{-.3cm}
During the preparation of this work, D. T. Das is highly grateful to the Birla Institute of Technology and Science, Pilani for providing the institute fellowship. S. Laishram was supported by ANRF CRG Grant CRG/2023/005564 and S. G. Nair was supported by the SERB MATRICS Grant MTR/2022/000864. The authors gratefully acknowledge ICDEPRA 2024, during which the problem investigated in this paper was conceived, and the collaboration leading to this work began.
\bibliographystyle{elsarticle-harv} 
\bibliography{Bibliography}

@book{polya_szego,
    author="G. P\`{o}lya and G. Szeg\`{o}",
    title="Problems and theorems in analysis. Vol. II. Theory of functions, zeros, polynomials, determinants, number theory, geometry. Die Grundlehren der Mathematischen Wissenschaften",
    publisher="Springer-Verlag, New York",
    year="1976"
}

@book{szego,
  title={Orthogonal Polynomials},
  author={Szeg{\H{o}}, G.},
  isbn={9780821810231},
  lccn={77476087},
  series={American Math. Soc: Colloquium publ},
  year={1975},
  publisher= {American Mathematical Society}
}

@book{Hasse,
  author = {H. Hasse},
  title = {Number Theory},
  series = {Classics in Mathematics},
  publisher = {Springer-Verlag},
  address = {Berlin},
  year = {2002}
}

@article{filaseta_finch_leidy,
    author =       "M. Filaseta and C. Finch and J. R. Leidy",
    title =        "{T. N. Shorey’s influence in the theory of irreducible polynomials}",
    journal =      "Diophantine Equations",
    publisher = "Narosa Publ., New Delhi",
    pages =        "77--102",
    year =         "2008"
}

@article{filaseta_all_but_finite_bessel,
    author = " M. Filaseta",
    title = "The irreducibility of all but finitely many Bessel polynomials",
    journal = "Acta Math, 174",
    year = "1995",
    pages = " 383--397"
}

@article{filaseta_kidd_trifonov,
  author = {M. Filaseta and T. Kidd and O. Trifonov},
  title = {Laguerre polynomials with Galois group $A_m$ for each $m$},
  journal = {J. Number Theory},
  volume = {132},
  year = {2012},
  pages = {776--805}
}

@article{filaseta_lam,
  author = {M. Filaseta and T. Y. Lam},
  title = {On the irreducibility of the generalized Laguerre polynomials},
  journal = {Acta Arith.},
  volume = {105},
  year = {2002},
  pages = {177--182}
}

@article{schur_1929_1,
    author="I. Schur",
    title=" Einige Sätze über Primzahlen mit Anwendungen auf Irreduzibilitätsfragen I",
    publisher="Sitzungsber. Preuss. Akad. Wiss.",
    journal="Berlin Phys.-Math. Kl.14",
    year="1929", 
    pages="125--136"
}

@article{schur_1931,
    author="I. Schur",
    title=" Affektlose Gleichungen in der Theorie der Laguerreschen und Hermiteschen Polynome",
    journal="Journal für die reine und angewandte Mathematik 165",
    year="1931", 
    pages="52--58"
}

@incollection{schur_1973_1,
  author = {I. Schur},
  title = {Gleichungen ohne Affekt},
  booktitle = {Gesammelte Abhandlungen, Band III},
  publisher = {Springer-Verlag},
  year = {1973},
  pages = {191--197}
}

@article{nair_shorey_GPF,
    author = "S. G. Nair and T. N. Shorey" ,
    title = "Lower bounds for the greatest prime factor of product of consecutive positive integers" ,
    journal = "Journal of Number Theory",
    year = "2015",
    pages   = {307--328}
}

@article{nair_shorey_survey,
  author  = {S. G. Nair and T. N. Shorey},
  title   = {Generalized Laguerre polynomials with applications},
  journal = {The Mathematics Student},
  volume  = {86},
  pages   = {87--101},
  year    = {2017}
}

@article{laishram_shorey_integer,
    author="S. Laishram and T. N. Shorey",
    title="  Extensions of Schur’s irreducibility results",
    journal="Indagationes Mathematicae 21 ",
    year="2011 ", 
    pages="87–105"
}

@article{laishram_shorey_GPF,
    author =" S. Laishram and T. N. Shorey", 
    title = " The greatest prime divisor of a product of consecutive integers",
    journal = "Acta Arith. 120",
    year = "2005",
    pages="299-306"
}

@article{lehmer,
    author = " D. H. Lehmer",
    title ="On a problem of St{\"o}rmer" ,
    journal ="Illinois J. of Math., 8" ,
    year = "1964",
    pages="57-79"
}

@article{dumas,
    author = "M. G. Dumas" ,
    title = "Sur quelques cas d'irréductibilité des polynômes à coefficients rationnels",
    journal = "J. Math. Pures Appl. {\bf 2}",
    year = "1906",
    pages = "191--258"
}

@article {banerjee,
    author = {P. Banerjee},
    title = {An unsolved question surrounding the generalized Laguerre
              polynomial {$L_n^{(n)}(x)$}},
   journal = {Ramanujan J.},
    volume = {65},
    year = {2024},
    number = {3},
    pages = {1147--1158},
}

@article{stormer,
  author  = {C. Störmer},
  title   = {Quelques th{\'e}or{\`e}mes sur l'{\'e}quation de Pell
             $x^{2} - D y^{2} = \pm 1$ et leurs applications},
  journal = {Skrifter Videnskabs-Selskabet i Christiania},
  series  = {Matematisk-Naturvidenskabelig Klasse},
  volume  = {I},
  number  = {2},
  year    = {1897},
  }

@article{lucas,
 author = {E. Lucas},
 journal = {American Journal of Mathematics},
 number = {2},
 pages = {184--196},
 publisher = {Johns Hopkins University Press},
 title = {Théorie des Fonctions Numériques Simplement Périodiques},
 volume = {1},
 year = {1878}
}

@article{najman,
 author = "F. Najman",
 title = "Large strings of consecutive smooth integers",
 fjournal = {Archiv der Mathematik},
 journal = "Arch. Math.",
 year = {2011},
 volume = {97},
 number = {4},
 pages = {319--324},
}

@article{gow,
  author = {R. Gow},
  title = {Some generalized Laguerre polynomials whose Galois groups are the Alternating groups},
  journal = {J. Number Theory},
  volume = {31},
  year = {1989},
  pages = {201--207}
}

\end{document}